\documentclass[oneside,10pt]{amsart}
\usepackage{amsthm,amssymb}
\usepackage{mathtools}
\usepackage[hidelinks]{hyperref}
\usepackage[a4paper,width=160mm,top=27mm,bottom=27mm]{geometry}
\numberwithin{equation}{section}
\usepackage{todonotes}
\usepackage{color}

\newtheorem{theorem}{Theorem}[section]
\newtheorem{lemma}[theorem]{Lemma}

\theoremstyle{definition}
\newtheorem{remx}[theorem]{Remark}
\newenvironment{remark}
  {\pushQED{\qed}\remx}
  {\popQED\endremx}

\begin{document}
\address{Hichem Hajaiej
\newline \indent Department of Mathematics, Cal State LA\indent
\newline \indent  Los Angeles CA 90032, USA.\indent }
\email{hhajaie@calstatela.edu}

\address{Yongming Luo
\newline \indent Institut f\"ur Wissenschaftliches Rechnen, Technische Universit\"at Dresden\indent
\newline \indent  Zellescher Weg 25, 01069 Dresden, Germany.\indent }
\email{yongming.luo@tu-dresden.de}

\address{Linjie Song
	\newline \indent Institute of Mathematics, AMSS, Chinese Academy of Science, Beijing 100190, China,\indent
	\newline \indent  University of Chinese Academy of Science, Beijing 100049, China.\indent }
\email{songlinjie18@mails.ucas.edu.cn}

\newcommand{\diver}{\operatorname{div}}
\newcommand{\lin}{\operatorname{Lin}}
\newcommand{\curl}{\operatorname{curl}}
\newcommand{\ran}{\operatorname{Ran}}
\newcommand{\kernel}{\operatorname{Ker}}
\newcommand{\la}{\langle}
\newcommand{\ra}{\rangle}
\newcommand{\N}{\mathbb{N}}
\newcommand{\R}{\mathbb{R}}
\newcommand{\C}{\mathbb{C}}
\newcommand{\T}{\mathbb{T}}

%%%%%%%%%%%%%%%%%%%%%%%%%%%%%%%%%%%%%%%%
\newcommand{\ld}{\lambda}
\newcommand{\fai}{\varphi}
\newcommand{\0}{0}
\newcommand{\n}{\mathbf{n}}
\newcommand{\uu}{{\boldsymbol{\mathrm{u}}}}
\newcommand{\UU}{{\boldsymbol{\mathrm{U}}}}
\newcommand{\buu}{\bar{{\boldsymbol{\mathrm{u}}}}}
\newcommand{\ten}{\\[4pt]}
\newcommand{\six}{\\[-3pt]}
\newcommand{\nb}{\nonumber}
\newcommand{\hgamma}{H_{\Gamma}^1(\OO)}
\newcommand{\opert}{O_{\varepsilon,h}}
\newcommand{\barx}{\bar{x}}
\newcommand{\barf}{\bar{f}}
\newcommand{\hatf}{\hat{f}}
\newcommand{\xoneeps}{x_1^{\varepsilon}}
\newcommand{\xh}{x_h}
\newcommand{\scaled}{\nabla_{1,h}}
\newcommand{\scaledb}{\widehat{\nabla}_{1,\gamma}}
\newcommand{\vare}{\varepsilon}
\newcommand{\A}{{\bf{A}}}
\newcommand{\RR}{{\bf{R}}}
\newcommand{\B}{{\bf{B}}}
\newcommand{\CC}{{\bf{C}}}
\newcommand{\D}{{\bf{D}}}
\newcommand{\K}{{\bf{K}}}
\newcommand{\oo}{{\bf{o}}}
\newcommand{\id}{{\bf{Id}}}
\newcommand{\E}{\mathcal{E}}
\newcommand{\ii}{\mathcal{I}}
\newcommand{\sym}{\mathrm{sym}}
\newcommand{\lt}{\left}
\newcommand{\rt}{\right}
\newcommand{\ro}{{\bf{r}}}
\newcommand{\so}{{\bf{s}}}
\newcommand{\e}{{\bf{e}}}
\newcommand{\ww}{{\boldsymbol{\mathrm{w}}}}
\newcommand{\zz}{{\boldsymbol{\mathrm{z}}}}
\newcommand{\U}{{\boldsymbol{\mathrm{U}}}}
\newcommand{\G}{{\boldsymbol{\mathrm{G}}}}
\newcommand{\VV}{{\boldsymbol{\mathrm{V}}}}
\newcommand{\II}{{\boldsymbol{\mathrm{I}}}}
\newcommand{\ZZ}{{\boldsymbol{\mathrm{Z}}}}
\newcommand{\hKK}{{{\bf{K}}}}
\newcommand{\f}{{\bf{f}}}
\newcommand{\g}{{\bf{g}}}
\newcommand{\lkk}{{\bf{k}}}
\newcommand{\tkk}{{\tilde{\bf{k}}}}
\newcommand{\W}{{\boldsymbol{\mathrm{W}}}}
\newcommand{\Y}{{\boldsymbol{\mathrm{Y}}}}
\newcommand{\EE}{{\boldsymbol{\mathrm{E}}}}
\newcommand{\F}{{\bf{F}}}
\newcommand{\spacev}{\mathcal{V}}
\newcommand{\spacevg}{\mathcal{V}^{\gamma}(\Omega\times S)}
\newcommand{\spacevb}{\bar{\mathcal{V}}^{\gamma}(\Omega\times S)}
\newcommand{\spaces}{\mathcal{S}}
\newcommand{\spacesg}{\mathcal{S}^{\gamma}(\Omega\times S)}
\newcommand{\spacesb}{\bar{\mathcal{S}}^{\gamma}(\Omega\times S)}
\newcommand{\skews}{H^1_{\barx,\mathrm{skew}}}
\newcommand{\kk}{\mathcal{K}}
\newcommand{\OO}{O}
\newcommand{\bhe}{{\bf{B}}_{\vare,h}}
\newcommand{\pp}{{\mathbb{P}}}
\newcommand{\ff}{{\mathcal{F}}}
\newcommand{\mWk}{{\mathcal{W}}^{k,2}(\Omega)}
\newcommand{\mWa}{{\mathcal{W}}^{1,2}(\Omega)}
\newcommand{\mWb}{{\mathcal{W}}^{2,2}(\Omega)}
\newcommand{\twos}{\xrightharpoonup{2}}
\newcommand{\twoss}{\xrightarrow{2}}
\newcommand{\bw}{\bar{w}}
\newcommand{\bz}{\bar{{\bf{z}}}}
\newcommand{\tw}{{W}}
\newcommand{\tr}{{{\bf{R}}}}
\newcommand{\tz}{{{\bf{Z}}}}
\newcommand{\lo}{{{\bf{o}}}}
\newcommand{\hoo}{H^1_{00}(0,L)}
\newcommand{\ho}{H^1_{0}(0,L)}
\newcommand{\hotwo}{H^1_{0}(0,L;\R^2)}
\newcommand{\hooo}{H^1_{00}(0,L;\R^2)}
\newcommand{\hhooo}{H^1_{00}(0,1;\R^2)}
\newcommand{\dsp}{d_{S}^{\bot}(\barx)}
\newcommand{\LB}{{\bf{\Lambda}}}
\newcommand{\LL}{\mathbb{L}}
\newcommand{\mL}{\mathcal{L}}
\newcommand{\mhL}{\widehat{\mathcal{L}}}
\newcommand{\loc}{\mathrm{loc}}
\newcommand{\tqq}{\mathcal{Q}^{*}}
\newcommand{\tii}{\mathcal{I}^{*}}
\newcommand{\Mts}{\mathbb{M}}
\newcommand{\pot}{\mathrm{pot}}
\newcommand{\tU}{{\widehat{\bf{U}}}}
\newcommand{\tVV}{{\widehat{\bf{V}}}}
\newcommand{\pt}{\partial}
\newcommand{\bg}{\Big}
\newcommand{\hA}{\widehat{{\bf{A}}}}
\newcommand{\hB}{\widehat{{\bf{B}}}}
\newcommand{\hCC}{\widehat{{\bf{C}}}}
\newcommand{\hD}{\widehat{{\bf{D}}}}
\newcommand{\fder}{\partial^{\mathrm{MD}}}
\newcommand{\Var}{\mathrm{Var}}
\newcommand{\pta}{\partial^{0\bot}}
\newcommand{\ptaj}{(\partial^{0\bot})^*}
\newcommand{\ptb}{\partial^{1\bot}}
\newcommand{\ptbj}{(\partial^{1\bot})^*}
\newcommand{\geg}{\Lambda_\vare}
\newcommand{\tpta}{\tilde{\partial}^{0\bot}}
\newcommand{\tptb}{\tilde{\partial}^{1\bot}}
\newcommand{\ua}{u_\alpha}
\newcommand{\pa}{p\alpha}
\newcommand{\qa}{q(1-\alpha)}
\newcommand{\Qa}{Q_\alpha}
\newcommand{\Qb}{Q_\eta}
\newcommand{\ga}{\gamma_\alpha}
\newcommand{\gb}{\gamma_\eta}
\newcommand{\ta}{\theta_\alpha}
\newcommand{\tb}{\theta_\eta}

%%%%%%%%%%%%%%%%%%%%%%%%%%

\newcommand{\mH}{{E}}
\newcommand{\mN}{{N}}
\newcommand{\mD}{{\mathcal{D}}}
\newcommand{\csob}{\mathcal{S}}
\newcommand{\mA}{{A}}
\newcommand{\mK}{{Q}}
\newcommand{\mS}{{S}}
\newcommand{\mI}{{I}}
\newcommand{\tas}{{2_*}}
\newcommand{\tbs}{{2^*}}
\newcommand{\tm}{{\tilde{m}}}
\newcommand{\tdu}{{\phi}}
\newcommand{\tpsi}{{\tilde{\psi}}}
\newcommand{\Z}{{\mathbb{Z}}}
\newcommand{\tsigma}{{\tilde{\sigma}}}
\newcommand{\tg}{{\tilde{g}}}
\newcommand{\tG}{{\tilde{G}}}
\newcommand{\mM}{{M}}
\newcommand{\mC}{\mathcal{C}}
\newcommand{\wlim}{{\text{w-lim}}\,}
\newcommand{\diag}{L_t^\ba L_x^\br}
\newcommand{\vu}{ u}
\newcommand{\vz}{ z}
\newcommand{\vv}{ v}
\newcommand{\ve}{ e}
\newcommand{\vw}{ w}
\newcommand{\vf}{ f}
\newcommand{\vh}{ h}
\newcommand{\vp}{ \vec P}
\newcommand{\ang}{{\not\negmedspace\nabla}}
\newcommand{\dxy}{\Delta_{x,y}}
\newcommand{\lxy}{L_{x,y}}
\newcommand{\gnsand}{\mathrm{C}_{\mathrm{GN},3d}}
\newcommand{\wmM}{\widehat{{M}}}
\newcommand{\wmH}{\widehat{{E}}}
\newcommand{\wmI}{\widehat{{I}}}
\newcommand{\wmK}{\widehat{{Q}}}
\newcommand{\wmN}{\widehat{{N}}}
\newcommand{\wm}{\widehat{m}}
\newcommand{\ba}{\mathbf{a}}
\newcommand{\bb}{\mathbf{b}}
\newcommand{\br}{\mathbf{r}}
\newcommand{\bq}{\mathbf{q}}
\newcommand{\SSS}{\mathcal{S}}
%-------------------------------------

%New Commands

%-------------------------------------

%%%%%%%%%%%%%%%%%%%%%%
\title[Existence and stability for ground states of BNLS]{On existence and stability results for normalized ground states of mass-subcritical biharmonic NLS on $\R^d\times\T^n$}
\author{Hichem Hajaiej, Yongming Luo and Linjie Song}

%\date{}
\maketitle

\begin{abstract}
We study the focusing mass-subcritical biharmonic nonlinear Schr\"odinger equation (BNLS) on the product space $\R_x^d\times\T_y^n$. Following the crucial scaling arguments introduced in \cite{TTVproduct2014} we establish existence and stability results for the normalized ground states of BNLS. Moreover, in the case where lower order dispersion is absent, we prove the existence of a critical mass number $c_0\in(0,\infty)$ that sharply determines the $y$-dependence of the deduced ground states. In the mixed dispersion case, we encounter a major challenge as the BNLS is no longer scale-invariant and the arguments from \cite{TTVproduct2014} for determining the sharp $y$-dependence of the ground states fail. The main novelty of the present paper is to address this difficult and interesting issue:
Using a different scaling argument, we show that $y$-independence of ground states with small mass still holds in the case $\beta>0$ and $\alpha\in(0,4/(d+n))$. Additionally, we also prove that ground states with sufficiently large mass must possess non-trivial $y$-dependence by appealing to some novel construction of test functions. The latter particularly holds for all parameters lying in the full mass-subcritical regime.

\end{abstract}

%%%%%%%%%%%%%%%%%%%%%%%%%%%%%

\section{Introduction and main results}\label{sec:Introduction and main results}
The paper is devoted to the study of the focusing mass-subcritical biharmonic nonlinear Schr\"odinger equation (BNLS)
\begin{align}\label{nls}
i\pt_t \phi-\Delta^2_{x,y}\phi+\beta\Delta_{x,y}\phi+|\phi|^\alpha u=0
\end{align}
on the product space $\R_x^d\times\T_y^n$ with $\beta\in\R$ and $\alpha\in(0,\frac{8}{d+n})$. The BNLS can be seen as a generalization of the standard focusing NLS
\begin{align}\label{standard_nls}
i\pt_t \phi+\Delta \phi+|\phi|^\alpha u=0
\end{align}
--- a prototype model arising in the study of nonlinear optics and Bose-Einstein condensates (see e.g. \cite{phy_double_crit_1,phy_double_crit_2} for related physical background) --- in the sense that higher order dispersive effects are also taken into account. To motivate the study of \eqref{nls}, one may consider the model \eqref{standard_nls} when the order $\alpha$ of the nonlinear potential term $|\phi|^\alpha \phi$ is sufficiently large (for instance when $\alpha$ lying in the mass-(super)critical regime $\alpha\geq \frac{4}{d}$ and \eqref{standard_nls} is considered on the Euclidean space $\R^d$). Then it is well-known that \eqref{standard_nls} may possess finite time blow-up solutions by appealing to the celebrated Glassey's virial identity \cite{Glassey1977} under suitable conditions on the initial data, see for instance \cite{Cazenave2003} for a rigorous proof of this phenomenon. However, in strong contrast to the rigorously derived finite time blow-up results, collapse does not take place in actual experiments in many cases. This shall indicate that a higher order stabilizing effect was neglected in \eqref{standard_nls} and a correction of standard NLS model becomes necessary. One possible way is to consider the modified model
\begin{align*}
i\pt_t \phi+\Delta \phi+|\phi|^\alpha \phi-|\phi|^\kappa \phi=0
\end{align*}
by incorporating a higher order repulsive potential $-|\phi|^\kappa \phi$ with $\kappa>\alpha$. This includes for instance the very famous cubic-quintic model ($\alpha=2$, $\kappa=4$, where the nonlinear potentials model the two- and three-body interactions respectively). In this direction, we refer e.g. to \cite{phy2,phy3,phy1,killip_visan_soliton,killip2020cubicquintic,Murphy2021CPDE} and the references therein for a more detailed survey.

Another possible extension of \eqref{standard_nls} was imposed by Karpman and Shagalov (see \cite{KarpmanShagalov2000} and the references therein), where the authors considered the model \eqref{nls} with bi-Laplacian $\Delta^2 \phi$ indicating a higher order dispersion effect. Different topics concerning the well-posedness and (in-)stability of ground states of the BNLS model \eqref{nls} have been nowadays intensively studied, see e.g. \cite{Pausader1,Pausader3,FibichIlan2002,Blowup4nls,MiaoXuZhao4nls,Bonheure1,Bonheure2,Bonheure3,LuoTingjian1,Remark2019} and the references therein for further details. Recently, there has been an increasing interest in studying dispersive equations on the waveguide manifolds $\R^d\times\T^n$ whose mixed type geometric nature makes the underlying analysis rather challenging and delicate. However, most of the results so far are established for the standard NLS model \cite{HaniPausader,ModifiedScattering,TTVproduct2014,TzvetkovVisciglia2016,RmT1,CubicR2T1Scattering,R1T1Scattering,Cheng_JMAA,similar_cubic,Luo_Waveguide_MassCritical,Luo_inter,Luo_energy_crit}. On the other hand, the first well-posedness and scattering result for \eqref{nls} without mixed dispersion (i.e. $\beta=0$) was only given in the very recent paper \cite{YuYueZhao4nlsfirst}. The purpose of this paper is therefore to study the existence and stability results for normalized ground states of \eqref{nls} on $\R^d\times\T^n$. To the best of our knowledge, there are so far no such results established in the literature. In particular, we also consider the mixed dispersion case where $\beta$ is not necessarily zero. Throughout the paper, we restrict ourselves to the mass-subcritical regime $\alpha\in(0,\frac{8}{d+n})$\footnote{The more complicated case $\alpha\in(\frac{8}{d+n},\infty)$, where blow-up and scattering results shall possibly be deduced, will eventually be studied in a forthcoming paper.}. Before we turn to our main results, the mass and energy quantities of \eqref{nls}, which will be frequently used throughout the paper, are defined as follows:
\begin{align*}
M(u)&=\int_{\R^d\times\T^n}|u(x,y)|^2\,dxdy,\\
E(u)&=\int_{\R^d\times\T^n}\frac{1}{2}|\Delta_{x,y}^2u(x,y)|^2+\frac{\beta}{2}|\nabla_{x,y}u(x,y)|^2-\frac{1}{\alpha+2}|u(x,y)|^{\alpha+2}\,dxdy.
\end{align*}
It is also well-known that both mass and energy are conserved over time along the BNLS flow \eqref{nls}.
\subsection{Main results}
To proceed, we firstly introduce the concept of a \textit{normalized ground state}. Inserting the standing wave ansatz $\phi(t,x,y)=e^{i\theta t}u(x,y)$ into \eqref{nls} and cancelling out the phase factor $e^{i\theta t}$, the following stationary BNLS follows:
\begin{align}\label{swnls}
\Delta_{x,y}^2 u-\beta \Delta_{x,y}u+\theta u=|u|^\alpha u.
\end{align}
In order to formulate our main results, we define
\begin{align*}
S(c)&:=\{u\in H^2(\R^d\times\T^n):M(u)=c\},\\
m_c&:=\inf\{E(u):u\in S(c)\},\\
\widehat{S}(c)&:=\{u\in H^2(\R^d):\widehat{M}(u)=c\},\\
\widehat{m}_c&:=\inf\{\widehat{E}(u):u\in \widehat{S}(c)\}.
\end{align*}
where the hatted quantities $\widehat{M}(u)$ and $\widehat{E}(u)$ are simply the mass and energy defined on $\R^d$, i.e.
\begin{align*}
\widehat{M}(u)&=\int_{\R^d}|u(x)|^2\,dx,\\
\widehat{E}(u)&=\int_{\R^d}\frac{1}{2}|\Delta_{x}^2u(x)|^2+\frac{\beta}{2}|\nabla_{x}u(x)|^2-\frac{1}{\alpha+2}|u(x)|^{\alpha+2}\,dx.
\end{align*}
Notice that by Lagrange multiplier theorem, any optimizer of $m_c$ will automatically be a solution of \eqref{swnls}. As optimizers of $m_c$ also possess least energy among all candidates and have fixed mass $c$, they will also be referred to as the normalized ground states. Since only normalized ground states will be considered in this paper, the term ``normalized'' will be simply omitted in the rest of the paper.

Our first result deals with the existence of ground state solutions of $m_c$ for a given mass number $c\in(0,\infty)$.

\begin{theorem}[Existence of ground states]\label{thm existence}
The following existence results hold:
\begin{itemize}
\item[(i)]Suppose that one of the following assumptions is satisfied:
\begin{itemize}
\item[(a)]$\beta=0$ and $\alpha\in(0,\frac{8}{d+n})$.
\item[(b)]$\beta> 0$ and $\alpha\in(0,\min\{\frac{8}{d+n},\frac{4}{d}\})$.
\end{itemize}
Then for any $c\in (0,\infty)$ the variational problem $m_c$ has an optimizer $u_c$.

\item [(ii)] Assume $\beta>0$, $d>n$ and $\alpha\in[\frac{4}{d},\frac{8}{d+n})$. Define
$$ c_+:=\inf\{c>0:m_c<0\}.$$
Then $c_+\in[0,\infty)$. Moreover, we have
\begin{itemize}
\item[(a)] For any $c\in(0,c_+)$ the variational problem $m_c$ has no optimizer.
\item[(b)] For any $c\in(c_+,\infty)$ the variational problem $m_c$ has an optimizer $u_c$.
\end{itemize}

\item[(iii)] Assume $\beta<0$ and $\alpha\in(0,\frac{8}{d+n})$. Define
$$ c_-:=\inf\{c>0:m_c<-\frac{\beta^2}{8}c\}.$$
Then $c_-\in[0,\infty)$. Moreover, we have
\begin{itemize}
\item[(a)] For any $c\in(0,c_-)$ the variational problem $m_c$ has no optimizer.
\item[(b)] For any $c\in(c_-,\infty)$ the variational problem $m_c$ has an optimizer $u_c$.
\end{itemize}
If additionally $\alpha\in(0,2)$, then $c_-=0$.
\end{itemize}
\end{theorem}

The proof of Theorem \ref{thm existence} is based on the classical concentration compactness arguments. In the context of a product space, we shall follow closely the steps from \cite{TTVproduct2014} to prove Theorem \ref{thm existence}. Other than the existence results, by the boundedness of $\T$ we obtain the interesting fact that if $u$ is constant along $y$-direction, \eqref{swnls} will reduce to the standard BNLS on the Euclidean space $\R^d$. It is therefore natural to ask whether the ground states deduced in Theorem \ref{thm existence} will either coincide with the ones on $\R^d$ (namely $y$-independent) or they have non-trivial $y$-portion. Depending on the choice of the parameters, we give a complete or partial answer to this question in our next result.

\begin{theorem}[$y$-dependence of the ground states]\label{thm dependence}
The following statements hold true:
\begin{itemize}
\item[(i)]
Assume that $\beta=0$ and $\alpha\in(0,\frac{8}{d+n})$. For any $c\in(0,\infty)$ let $u_c$ be a minimizer of $m_c$ deduced from Theorem \ref{thm existence}. Then there exists some $c_0\in (0,\infty)$ such that $m_{c_0}=(2\pi)^n\widehat{m}_{(2\pi)^{-n}c_0}$. Moreover,
\begin{itemize}
\item[(a)] For any $c\in(0,c_0)$ we have $m_c=(2\pi)^n\widehat{m}_{(2\pi)^{-n}c}$. Particularly, $\nabla_y u_c= 0$.
\item[(b)] For any $c\in(c_0,\infty)$ we have $m_c<(2\pi)^n\widehat{m}_{(2\pi)^{-n}c}$. Particularly, $\nabla_y u_c\neq 0$.
\end{itemize}

\item[(ii)]
Assume that $\beta > 0$ and $\alpha\in(0,\frac{4}{d+n})$. For any $c\in(0,\infty)$ let $u_c$ be a minimizer of $m_c$ deduced from Theorem \ref{thm existence}. Then there exists some $c_{>0}\in (0,\infty)$ such that for any $c\in(0,c_{>0}]$ we have {$m_{c}=(2\pi)^n\widehat{\mu}^{1}_{(2\pi)^{-n}c}$} and for any $c\in(0,c_{>0})$ it holds $\nabla_y u_c= 0$. Here, the quantity {$\widehat{\mu}^1_c$} is defined by \eqref{A.16} below.

\item[(iii)]
Assume that $\beta\neq0$ and $\alpha\in(0,\frac{8}{d+n})$. Then there exists some $c_{\neq 0}\in[0,\infty)$ such that $m_{c}$ has an optimizer $u_c$ for all $c\in(c_{\neq 0},\infty)$ and $\nabla_y u_c\neq 0$.
\end{itemize}
\end{theorem}

Again, we borrow a crucial scaling argument from \cite{TTVproduct2014} to prove Theorem \ref{thm dependence} (i). More precisely, in the case $\beta=0$, by a simple scaling argument, proving Theorem \ref{thm dependence} (i) is essentially equivalent to showing that
\begin{align}\label{ld conv}
\lim_{\ld\to \infty}m_{1,\ld}=(2\pi)^n\wm_{(2\pi)^{-n}},\quad\lim_{\ld\to 0}m_{1,\ld}<(2\pi)^n\wm_{(2\pi)^{-n}},
\end{align}
where the modified energy $\mH_{\ld}(u)$ is defined by \eqref{def modified energy} and the minimization problem $m_{1,\ld}$ is given by
\begin{align*}
m_{1,\ld}:=\inf_{u\in H^1(\R^d\times\T)}\{\mH_{\ld}(u):u\in S(1)\}.
\end{align*}
We also underline that the proof of Theorem \ref{thm dependence} (i) is almost identical to the one of \cite[Thm. 1.3]{TTVproduct2014}. Nonetheless, when $\beta\neq 0$ we encounter a major challenge as \eqref{swnls} is no longer scaling-invariant. Moreover, even if we assume that a function $u\in S(1)$ is independent of $y$, the modified energy $E_{\ld}(u)$ will still remain $\ld$-dependent and \eqref{ld conv} can no longer hold.
As we shall see from the proofs of the main results, we will overcome this difficulty by using some subtle scaling arguments and by constructing certain delicate test functions. It remains an interesting problem whether $c_{>0}$ and $c_{\neq 0}$ coincide in the second case of Theorem \ref{thm dependence}. We are however unable to answer this question at this moment and thus leave it open in this article.

%\textcolor{blue}{The main novelty of this paper is to address this important and difficult issue. Using a different scaling argument, we show that $y$-independence of ground states with a small mass holds in the mixed dispersion case ($\beta > 0$, $\alpha\in(0,\frac{4}{d+n})$). We also show $y$-dependence of ground states with sufficiently large mass and the proof relies on some novel construction of test functions ($\beta\neq0$, $\alpha\in(0,\frac{8}{d+n})$). It is an interesting problem whether $c_1 = c_{\neq 0}$ or not when $\beta > 0$ and $\alpha\in(0,\frac{4}{d+n})$. In other words, can the loss of scaling invariance lead to $y$-independence and $y$-dependence appear alternately as $c$ changes? We leave this question open in this article.}

Finally, due to the mass-subcritical nature of \eqref{nls} and \eqref{swnls} we are able to prove the following stability result for the ground states obtained from Theorem \ref{thm existence}.

\begin{theorem}[Orbital stability]\label{thm stability}
Let $c_+,c_-$ be the numbers defined in Theorem \ref{thm existence} and define
$$ \Gamma_c:=\{u\in S(c):E(u)=m_c\}.$$
Assume that one of the following condition is satisfied:
\begin{itemize}
\item $\beta=0$, $\alpha\in(0,\frac{8}{d+n})$, $c\in(0,\infty)$.
\item $\beta>0$, $\alpha\in(0,\min\{\frac{8}{d+n},\frac{4}{d}\})$, $c\in(0,\infty)$.
\item $\beta>0$, $d>n$, $\alpha\in[\frac{4}{d},\frac{8}{d+n})$, $c\in(c_+,\infty)$.
\item $\beta<0$, $d>n$, $\alpha\in(0,\frac{8}{d+n})$, $c\in(c_-,\infty)$.
\end{itemize}
Assume also that \eqref{nls} is globally well-posed for any initial data lying in a neighborhood $\mathcal{U}$ of $\Gamma_c$. Then the set $\Gamma_c$ is orbitally stable in the sense that for all $\vare>0$ there exists some $\delta=\delta(\vare)>0$ such that for any $\phi_0\in \mathcal{U}$ satisfying
$$ \inf_{u\in \Gamma_c}\|\phi_0-u\|_{H^2(\R^d\times\T^n)}<\delta,$$
we have
$$\sup_{t\in\R}\inf_{u\in \Gamma_c}\|\phi(t)-u\|_{H^2(\R^d\times\T^n)}<\vare,$$
where $\phi$ is the global solution of \eqref{nls} with $\phi(0)=\phi_0$.
\end{theorem}

\begin{remark}
In the statement of Theorem \ref{thm stability} we have simply assumed that \eqref{nls} is globally well-posed in a neighborhood of $\Gamma_c$. However, whether this assumption holds in fact still remains an open problem. For a partial answer to this question, we point out that this assumption was confirmed in \cite{YuYueZhao4nlsfirst} in the sense that when $d\geq 5$, $n\leq 3$, $\beta=0$, $\alpha\in(0,\frac{8}{d+n})$ and $c\in(0,\infty)$, \eqref{nls} is globally well-posed in $H^2(\R^d\times\T^n)$. A further study on the verification of the assumption seems to be out of the scope of this paper and we hence leave this problem open for future research.
\end{remark}

\subsection{Notation and definitions}
For simplicity, we ignore in most cases the dependence of the function spaces on their underlying domains and hide this dependence in their indices. For example $L_x^2=L^2(\R^d)$, $H_{x,y}^2= H^2(\R^d\times \T^n)$ and so on. Moreover, the norm $\|\cdot\|_{L_{x,y}^p}$ will be abbreviated to $\|\cdot\|_p$. For our purpose we will also make use of the following sets:
\begin{gather}
C_{\rm per}^\infty(\T^n):=\{u\in C^\infty(\R^n),\,\forall\,y\in\R^n:u(y+(2\pi)^d)=u(y)\},\\
C_c^\infty(\R^d)\otimes C_{\rm per}^\infty(\T^d):=\{\sum_{i=1}^n a_ib_i:n\in\N,a_i\in C_c^\infty(\R^d),b_i\in C_{\rm per}^\infty(\T^d) \}.\label{tensor space}
\end{gather}

Next, we define the quantities such as mass and energy etc. that will be frequently used in the proof of the main results. For $u\in H_{x,y}^2$ and $\ld\in(0,\infty)$, $\tau\in(0,\infty)$, define
\begin{align}
\mM(u)&:=\|u\|^2_{2},\label{def of mass}\\
\mH(u)&:=\frac{1}{2}\|\Delta_{x,y} u\|^2_{2}+\frac{\beta}{2}\|\nabla_{x,y} u\|^2_{2}-\frac{1}{\alpha+2}\|u\|^{\alpha+2}_{\alpha+2}\label{def of mhu},\\
\mH_{\ld}(u)&:=\frac{\ld}{2}\|\Delta_{y} u\|^2_{2}+\frac{\beta\ld}{2}\|\nabla_{y} u\|^2_{2}+\frac{1}{2}\|\Delta_{x} u\|^2_{2}+\frac{\beta\sqrt{\ld}}{2}\|\nabla_{x} u\|^2_{2}
-\frac{1}{\alpha+2}\|u\|^{\alpha+2}_{\alpha+2}\label{def modified energy},\\
\mH^{\tau}(u)&:=\frac{\tau}{2}\|\Delta_{y} u\|^2_{2}+\frac{\beta\tau}{2}\|\nabla_{y} u\|^2_{2}+\frac{1}{2\tau}\|\Delta_{x} u\|^2_{2}+\frac{\beta}{2}\|\nabla_{x} u\|^2_{2}
-\frac{1}{\alpha+2}\|u\|^{\alpha+2}_{\alpha+2}\label{def modified energy 2}.
\end{align}
For $u\in H_x^2$ and $\ld\in(0,\infty)$, $\tau\in(0,\infty)$, define
\begin{align}
\wmM(u)&:=\|u\|^2_{L_x^2},\\
\wmH(u)&:=\frac{1}{2}\|\Delta_{x} u\|^2_{L_x^2}+\frac{\beta}{2}\|\nabla_{x} u\|^2_{L_x^2}-\frac{1}{\alpha+2}\|u\|^{\alpha+2}_{L_x^{\alpha+2}},\\
\widehat{E}_\ld(u)&:=\frac12\|\Delta_x u\|_{L_x^2}^2+\frac{\beta\sqrt{\ld}}{2}\|\nabla_x u\|_{L_x^2}^2
-\frac{1}{\alpha+2}\|u\|_{L_x^{\alpha+2}}^{\alpha+2},\\
\widehat{E}^\tau(u)&:=\frac{1}{2\tau}\|\Delta_x u\|_{L_x^2}^2+\frac{\beta}{2}\|\nabla_x u\|_{L_x^2}^2
-\frac{1}{\alpha+2}\|u\|_{L_x^{\alpha+2}}^{\alpha+2},\\
\widehat{E}^{\infty}(u)&:=\frac{\beta}{2}\|\nabla_x u\|_{L_x^2}^2
-\frac{1}{\alpha+2}\|u\|_{L_x^{\alpha+2}}^{\alpha+2}.
\end{align}
We also define the sets
\begin{align}
S(c)&:=\{u\in H_{x,y}^2:\mM(u)=c\},\\
\widehat{S}(c)&:=\{u\in H_x^2:\wmM(u)=c\}
\end{align}
and the variational problems
\begin{align}
m_c&:=\inf\{\mH(u):u\in S(c)\},\label{def of mc}\\
m_{1,\ld}&:=\inf\{\mH_\ld(u):u\in S(1)\},\label{def of auxiliary problem}\\
\mu_{1,\tau}&:=\inf\{\mH^\tau(u):u\in S(1)\},\label{def of auxiliary problem 2}\\
\wm_c&:=\inf\{\wmH(u):u\in \widehat{S}(c)\}\label{def of wmc},\\
\widehat{m}_c^\ld&:=\inf\{\widehat{E}_\ld(u):u\in\widehat{S}(c)\},\label{1.16}\\
\widehat{\mu}_c^\tau&:=\inf\{\widehat{E}^\tau(u):u\in\widehat{S}(c)\},\label{A.16}\\
\widehat{\mu}_c^{\infty}&:=\inf\{\widehat{E}^{\infty}(u):u\in\widehat{S}(c)\}.
\end{align}
For later use, we also remark that for $\beta>0$ and $\alpha<\frac{4}{d}$ we in fact have
\begin{align}\label{equalitya}
 \widehat{\mu}_c^{\infty}=\inf\{\widehat{E}^{\infty}(u):u\in H_x^1,\widehat{M}(u)=c\}=:\tilde{\mu}_c^{\infty}
\end{align}
and $\widehat{\mu}_c^{\infty}$ has at least an optimizer. To see this, it is trivial that $\widehat{\mu}_c^{\infty}\geq \tilde{\mu}_c^{\infty}$ since $H_x^2\subset H_x^1$. Moreover, it is well-known (see for instance \cite{Cazenave2003}) that $\tilde{\mu}_c^{\infty}$ has an optimizer $P\in W_x^{3,p}$ for all $p\in[2,\infty)$, which in turn implies \eqref{equalitya} and the existence of optimizers of $\widehat{\mu}_c^{\infty}$.
\section{Some preliminaries}
Before giving the proofs of our main results, we first collect some useful auxiliary lemmas. The first tool will be an inhomogeneous Gagliardo-Nirenberg inequality on $\R^d\times\T^n$.

\begin{lemma}[Gagliardo-Nirenberg inequality]\label{lem:gn_ineq}
On $\R^d\times\T^n$ we have the Gagliardo-Nirenberg inequality
\begin{align*}
\|u\|_{\alpha+2}^{\alpha+2}\leq C \|u\|^{\frac{\alpha(d+n)}{4}}_{H^2_{x,y}}\|u\|_2^{\alpha+2-\frac{\alpha(d+n)}{4}},
\end{align*}
where
\begin{align*}
\left\{
             \begin{array}{ll}
             \alpha\in(0,\infty),&\text{if $d+n\leq 4$},\\
             \alpha\in(0,\frac{4}{d+n-4}),&\text{if $d+n> 4$}.
             \end{array}
\right.
\end{align*}
\end{lemma}

\begin{proof}
Let $\phi\in C_c^\infty(\R^n;[0,1])$ be a radially symmetric and decreasing cut-off function such that $\phi\equiv 1$ on $[-2\pi,2\pi]^n$ and $\phi(z)\equiv 0$ for $z\in \R^n\setminus [-4\pi,4\pi]^n$. Let $v(x,y):=u(x,y)\phi(y)$. Then for $u\in H_{x,y}^2$ we have $v(x,y)\in H^2(\R^{d+n})$. Thus using the Gagliardo-Nirenberg inequality on $\R^{d+n}$ and followed by product rule we obtain
\begin{align}
\|u\|_{\alpha+2}^{\alpha+2}\leq \|v\|_{L^{\alpha+2}(\R^{d+n})}^{\alpha+2}
\lesssim
\|\Delta_{\R^{d+n}}v\|^{\frac{\alpha(d+n)}{2}}_{L^2(\R^{d+n})}\|v\|_{L^2(\R^{d+n})}^{\alpha+2-\frac{\alpha(d+n)}{2}}
\lesssim \|u\|^{\frac{\alpha(d+n)}{2}}_{H^2_{x,y}}\|u\|_2^{\alpha+2-\frac{\alpha(d+n)}{2}}.
\end{align}
\end{proof}

Next we state some properties of the minimization problem $\widehat{m}_c$ considered on $\R^d$.
\begin{theorem}[Properties of $\widehat{m}_c$, \cite{Bonheure1,Remark2019}]\label{1thm nongative}
The following statements hold true:
\begin{itemize}
\item[(i)] Suppose that one of the following assumptions is satisfied:
\begin{itemize}
\item[(a)]$\beta=0$ and $\alpha\in(0,\frac{8}{d})$.
\item[(b)]$\beta>0$ and $\alpha\in(0,\frac{4}{d})$.
\end{itemize}
Then for any $c\in(0,\infty)$ we have $\widehat{m}_c\in(-\infty,0)$ and $\widehat{m}_c$ has a minimizer $U_c\in \widehat{S}(c)$.

\item[(ii)] Assume $\beta> 0$ and $\alpha\in[\frac{4}{d},\frac{8}{d})$. Then there exists some $\widehat{c}_{+}\in(0,\infty)$ such that
\begin{itemize}
\item[(a)]For any $c\in(0,\widehat{c}_{+})$ we have $\widehat{m}_c=0$ and $\widehat{m}_c$ has no minimizer.
\item[(b)]For any $c\in(\widehat{c}_{+},\infty)$ we have $\widehat{m}_c\in(-\infty,0)$ and $\widehat{m}_c$ has a minimizer $U_c\in \widehat{S}(c)$.
\end{itemize}

\item[(iii)] Assume $\beta< 0$ and $\alpha\in(0,\frac{8}{d})$. Then there exists some $\widehat{c}_{-}\in[0,\infty)$ such that
\begin{itemize}
\item[(a)]For any $c\in(0,\widehat{c}_{-})$ we have $\widehat{m}_c=-\frac{\beta^2c}{8} $ and $\widehat{m}_c$ has no minimizer.
\item[(b)]For any $c\in(\widehat{c}_{-},\infty)$ we have $\widehat{m}_c\in(-\infty,-\frac{\beta^2c}{8})$ and $\widehat{m}_c$ has a minimizer $U_c\in \widehat{S}(c)$.
\end{itemize}
If additionally $\alpha\in(0,\max\{\frac{8}{d+1},2\})$, then $\widehat{c}_{-}=0$.

\end{itemize}
\end{theorem}

The following lemma is crucial for identifying a non-vanishing weak limit of a minimizing sequence of $m_c$.
\begin{lemma}[Concentration compactness]\label{lem:non-vanishing weak limit}
Let $(u_n)_n$ be a bounded sequence in $H^2_{x,y}$. Let also
\begin{align*}
\left\{
             \begin{array}{ll}
             \alpha\in(0,\infty),&\text{if $d+n\leq 4$},\\
             \alpha\in(0,\frac{4}{d+n-4}),&\text{if $d+n> 4$}.
             \end{array}
\right.
\end{align*}
Assume that
\begin{align}\label{liminf pos}
\liminf_{n\to\infty}\|u_n\|_{\alpha+2}>0.
\end{align}
Then up to a subsequence, there exist $(x_n)_n\subset\R^d$ and some $u\in H^2_{x,y}\setminus\{0\}$ such that
$$ u_n(x+x_n,y)\rightharpoonup u(x,y)\quad\text{in $H^2_{x,y}$}.$$
\end{lemma}

\begin{proof}
We follow the same lines as in \cite{TTVproduct2014} to prove the claim. For $x\in\R^d$ define
$$ Q_x:=x+[0,1)^d\subset\R^d.$$
Then $\R^d=\dot{\cup}_{z\in\Z^d}Q_z$. Using the Gagliardo-Nirenberg inequality on $Q_z$ (which can be similarly proved as the one given by Lemma \ref{lem:gn_ineq}) we infer that
\begin{align}
\|u_n\|_{L^{2+\frac{8}{d+n}}(Q_z\times\T^n)}^{2+\frac{8}{d+n}}\leq C
\|u_n\|_{L^2(Q_z\times\T^n)}^{\frac{8}{d+n}}\|u_n\|_{H^2(Q_z\times\T^n)}^{2},\label{gn_on_qz}
\end{align}
where $C$ is some constant which can be uniformly chosen for all $Q_z$, since all $Q_z$ have same shapes. Summing \eqref{gn_on_qz} over $z\in\R^d$ and estimating  $\|u_n\|_{L^2(Q_z\times\T^n)}$ by $\sup_{z\in\Z^d}\|u_n\|_{L^2(Q_z\times\T^n)}$ we obtain
\begin{align}
\|u_n\|_{{2+\frac{8}{d+n}}}^{2+\frac{8}{d+n}}\leq C
(\sup_{z\in\Z^d}\|u_n\|_{L^2(Q_z\times\T^n)}^{\frac{8}{d+n}})\|u_n\|_{H^2_{x,y}}^{2}.\label{gn_on_qz2}
\end{align}
Thus combining with the $H^2_{x,y}$-boundedness of $(u_n)_n$, \eqref{liminf pos} and interpolation of Lebesgue norms we deduce that up to a subsequence of $(u_n)_n$ and a sequence $(x_n)_n\subset\Z^d$ such that for $v_n(x,y):=u_n(x+x_n,y)$ we have
\begin{align}
1\lesssim \sup_{n\in\N}\|v_n\|_{L^2(Q_0\times\T^n)}.
\end{align}
The desired claim then follows from Rellich's compact embedding applied on $Q_0\times\T^n$.
\end{proof}

\section{Proof of Theorem \ref{thm existence}}
\subsection{Proof of Theorem \ref{thm existence}, Case (i)}
\begin{proof}[Proof of Theorem \ref{thm existence}, Case (i)]
The proof is based on standard concentration compactness arguments. In the context of product spaces, we follow the same steps from \cite{TTVproduct2014} to prove the claim. We split our proof into four steps.

\subsubsection*{Step 1: Existence of a bounded minimizing sequence and non-vanishing weak limit}
Notice that in Case (i) we always have $\alpha<4/d$. By assuming that a candidate in $S(c)$ is independent of $y$ we already have $m_c\leq (2\pi)^n\widehat{m}_c<0$, where the negativity of $\widehat{m}_c$ is deduced from Theorem \ref{1thm nongative}. Let $(u_n)_n\subset S(c)$ be a minimizing sequence of $m_c$. Using Lemma \ref{lem:gn_ineq} we infer that
\begin{align}
m_c+c+o_n(1)&\geq c+\frac{1}{2}\|\Delta_{x,y} u\|_2^2+\frac{\beta}{2}\|\nabla_{x,y}u_n\|_2^2-\frac{1}{\alpha+2}\|u_n\|_{\alpha+2}^{\alpha+2}\nonumber\\
&\gtrsim \|u_n\|_{H_{x,y}^2}^2-Cc^{\frac{\alpha+2}{2}-\frac{\alpha(d+n)}{4}}\|u_n\|^{\frac{\alpha(d+n)}{4}}_{H_{x,y}^2}\nonumber\\
&\geq \inf_{t>0}(t^2-Cc^{\frac{\alpha+2}{2}-\frac{\alpha(d+n)}{4}}t^{\frac{\alpha(d+n)}{4}}) >-\infty,\label{mc larger than minus inf}
\end{align}
where in the last inequality we used the fact that $\alpha\in(0,8/(d+n))$. This in turn implies that $m_c>-\infty$. Finally, if $\|u_n\|_{H^2_{x,y}}\to \infty$, then using the second inequality of \eqref{mc larger than minus inf} we deduce the contradiction $m_c\geq \infty$. We finally prove a non-vanishing weak limit of a minimizing sequence. By definition of $m_c$ we have
\begin{align}\label{2.7}
\|u_n\|_{\alpha+2}^{\alpha+2}=(\alpha+2)(-m_c+o_n(1)+\frac{1}{2}\|\Delta_{x,y}u_n\|_2^2+\frac{1}{2}\|\nabla_{x,y}u_n\|_2^2)
\geq (\alpha+2)(-m_c+o_n(1)).
\end{align}
Thus combining with $m_c<0$ we infer that $\liminf_{n\to\infty}\|u_n\|_{\alpha+2}>0$. The claim then follows from Lemma \ref{lem:non-vanishing weak limit}. This completes the proof of Step 1.

\subsubsection*{Step 2: Continuity of the mapping $c\mapsto m_c$ on $(0,\infty)$}
Let $(u_n)_n$ be a bounded minimizing sequence of $m_c$ (whose existence is guaranteed by Step 1). Let $(c_j)_j$ be a positive sequence with $\lim_{j\to\infty}c_j=c$. Then by definition of $m_c$ we infer that
\begin{align}
m_{c_j}&\leq E(\frac{\sqrt{c_j}}{\sqrt{c}}u_n)=\frac{c_j}{c}\bg(\frac{1}{2}\|\Delta_{x,y}u_n\|_2^2+\frac{\beta}{2}\|\nabla_{x,y}u_n\|_2^2
-\frac{1}{\alpha+2}\bg(\frac{c_j}{c}\bg)^{\frac{\alpha}{2}}\|u_n\|_{\alpha+2}^{\alpha+2}\bg)\nonumber\\
&=\frac{c_j}{c}\bg(\frac{1}{2}\|\Delta_{x,y}u_n\|_2^2+\frac{\beta}{2}\|\nabla_{x,y}u_n\|_2^2
-\frac{1}{\alpha+2}\|u_n\|_{\alpha+2}^{\alpha+2}\bg)
+\frac{c_j}{c}\bg(\bg(1-\bg(\frac{c_j}{c}\bg)^{\frac{\alpha}{2}}\bg)\bg)\frac{1}{\alpha+2}\|u_n\|_{\alpha+2}^{\alpha+2}\nonumber\\
&=E(u_n)+\bg(\frac{c_j}{c}-1\bg)\bg(\frac{1}{2}\|\Delta_{x,y}u_n\|_2^2+\frac{\beta}{2}\|\nabla_{x,y}u_n\|_2^2
-\frac{1}{\alpha+2}\|u_n\|_{\alpha+2}^{\alpha+2}\bg)\nonumber\\
&+\frac{c_j}{c}\bg(\bg(1-\bg(\frac{c_j}{c}\bg)^{\frac{\alpha}{2}}\bg)\bg)\frac{1}{\alpha+2}\|u_n\|_{\alpha+2}^{\alpha+2}=:E(u_n)+I+II.\label{2.8}
\end{align}
Using $c_j= c+o_j(1)$ and the boundedness of $u_n$ in $H_{x,y}^2$ we obtain that $I,II\to 0$ as $j\to\infty$. Since $(u_n)_n$ is a minimizing sequence for $m_c$, by combining standard diagonal arguments we conclude that
\begin{align}
\limsup_{j\to\infty}\,m_{c_j}\leq m_c.
\end{align}
Thus it is left to prove the opposite inequality. Let $u_j\in S(c_j)$ satisfy $E(u_j)\leq m_{c_j}+j^{-1}$. Using \eqref{mc larger than minus inf} we know that $(u_j)_j$ is a bounded sequence in $H_{x,y}^2$. Using the definition of $m_c$ and arguing as in \eqref{2.7} we see that
\begin{align}
m_c&\leq E(\frac{\sqrt{c}}{\sqrt{c_j}}u_j)\nonumber\\
&\leq m_{c_j}+j^{-1}+\bg(\frac{c}{c_j}-1\bg)\bg(\frac{1}{2}\|\Delta_{x,y}u_j\|_2^2+\frac{\beta}{2}\|\nabla_{x,y}u_j\|_2^2
-\frac{1}{\alpha+2}\|u_j\|_{\alpha+2}^{\alpha+2}\bg)\nonumber\\
&+\frac{c}{c_j}\bg(\bg(1-\bg(\frac{c}{c_j}\bg)^{\frac{\alpha}{2}}\bg)\bg)\frac{1}{\alpha+2}\|uj\|_{\alpha+2}^{\alpha+2}.
\end{align}
Taking $j\to\infty$ we then similarly deduce that
\begin{align}
m_c\leq\liminf_{j\to\infty}m_{c_j}.
\end{align}
This completes the proof of Step 2.

\subsubsection*{Step 3: The mapping $c\mapsto c^{-1}m_c$ is strictly decreasing on $(0,\infty)$}
Let $0<c_1<c_2<\infty$. Let $(u_n)$ be a minimizing sequence of $m_{c_1}$ with $\liminf_{n\to\infty}\|u_n\|_{\alpha+2}^{\alpha+2}>0$ (whose existence is guaranteed by Step 1). Then arguing as in \eqref{mc larger than minus inf} we obtain
\begin{align}
m_{c_2}\leq \frac{c_2}{c_1}(m_{c_1}+o_n(1))+\frac{c_2}{c_1}\bg(1-\bg(\frac{c_2}{c_1}\bg)^{\frac{\alpha}{2}}\bg)\frac{1}{\alpha+2}\|u_n\|_{\alpha+2}^{\alpha+2}.
\end{align}
Thus combining with  $\liminf_{n\to\infty}\|u_n\|_{\alpha+2}^{\alpha+2}>0$ and taking $n\to\infty$ yields
\begin{align}
m_{c_2}<\frac{c_2}{c_1}m_{c_1}.
\end{align}
This completes the proof of Step 3.

\subsubsection*{Step 4: Conclusion}
We finish our proof in this final step. Let $(u_n)_n$ be a minimizing sequence of $m_c$, which possesses a non-vanishing weak limit $u\in H_{x,y}^2\setminus\{0\}$ (whose existence is guaranteed by Step 1). Using weakly lower semicontinuity of norms it is necessary that $\|u\|_2^2\in(0,c]$. We hence show that $\|u\|_2^2=:c_1\in(0,c)$ shall lead to a contradiction, which in turn implies that $\|u\|_2^2=c$ and the desired proof follows. Using Br\'{e}zis-Lieb lemma and the fact that $L_{x,y}^2$ is a Hilbert space we have
\begin{align}
E(u_n-u)+E(u)&=E(u_n)+o_n(1),\label{214}\\
M(u_n-u)+M(u)&=M(u_n)+o_n(1).\label{215}
\end{align}
By \eqref{215} we know that $M(u_n-u)=c-c_1+o_n(1)$. Hence by the definition of $m_{c-c_1+o_n(1)}$
$$E(u_n-u)\geq m_{c-c_1+o_n(1)}.$$
Taking $n\to\infty$ in \eqref{214} and using the continuity of the mapping $c\mapsto m_c$ (Step 2) we infer that
\begin{align}
m_c=E(u)+\lim_{n\to\infty}E(u_n-u)\geq m_{c_1}+m_{c-c_1}.\label{216}
\end{align}
Therefore combining this with Step 3 and $c_1,c-c_1<c$ we obtain
\begin{align}
m_c>\frac{c_1}{c}m_c+\frac{c-c_1}{c}m_c=m_c,
\end{align}
a contradiction. This completes the desired proof.
\end{proof}

\subsection{Proof of Theorem \ref{thm existence}, Case (ii)}
\begin{proof}[Proof of Theorem \ref{thm existence}, Case (ii)]
The proof in Case (ii) is in fact very similar to the Case (i). However, in this case we do not necessarily have $\widehat{m}_c\in(-\infty,0)$ for all $c\in(0,\infty)$. Thus we define
$$c_+:=\inf\{c\in(0,\infty):m_c\in(-\infty,0)\}.$$
By Theorem \ref{1thm nongative} we infer that $c_+\in[0,\infty)$. Hence the existence of optimizers of $m_c$ for $c\in(c_+,\infty)$ follows already from the previous proof of Theorem \ref{thm existence} in Case (i). It remains to show that $m_c$ has no optimizers for $c\in(0,c_+)$. First we prove
\begin{align}
m_{c_1+c_2}\leq m_{c_1}+m_{c_2}\label{subadditivity}
\end{align}
for any $c_1,c_2\in(0,\infty)$. To see this, by the definition of $m_{c}$ and density arguments, for any $\vare>0$ we can find
$u_i\in C_c^\infty(\R^d)\otimes C^\infty_{\rm per}(\T^n)$ (recall that the space $C_c^\infty(\R^d)\otimes C^\infty_{\rm per}(\T^n)$ is defined by \eqref{tensor space}) such that $M(u_i)=c_i$ and $m_{c_i}\leq E(u_i)+\vare$ for $i\in\{1,2\}$. Since $E(u)$ is invariant under $\R^d$-translation, we may also assume that $u_1,u_2$ have compact support. Thus
$$m_{c_1+c_2}\leq E(u_1+u_2)=E(u_1)+E(u_2)\leq m_{c_1}+m_{c_2}+2\vare.$$
\eqref{subadditivity} follows by taking $\vare\to0$. By combining $m_c\leq 0$ we infer that $c\mapsto m_c$ is monotone decreasing on $(0,\infty)$. Since by the definition and continuity of the mapping $c\mapsto m_c$ (deduced from Step 2 in the previous proof), we know that $m_c=0$ for all $c\in(0,c_+]$. Now let $b\in(0,c_+)$. We claim that if $m_b$ has an optimizer $u$, then $m_{\tau}<m_b$ for all $\tau\in(b,\infty)$. Let $u_{\tau}(x,y):=u((\tau/b)^{-\frac1d}x,y)$. Then $M(u_\tau)=\tau$ and using $\tau/b>1$
\begin{align}\label{2.19}
m_\tau\leq E(u_\tau)=\bg(\frac{\tau}{b}\bg)\bg(\bg(\frac{\tau}{b}\bg)^{-\frac{4}{d}}\frac12\|\Delta_{x,y}u\|_2^2
+\bg(\frac{\tau}{b}\bg)^{-\frac{2}{d}}\frac{\beta}{2}\|\nabla_{x,y}u\|_2^2
-\|u\|_{\alpha+2}^{\alpha+2}\bg)< \bg(\frac{\tau}{b}\bg)m_{b}.
\end{align}
This in turn implies that $m_\tau<(\tau/b)m_c\leq m_c$, since $m_b\leq 0$ and $\tau/b>1$, which also contradicts the fact that $m_c=0$ for all $c\in(0,c_+]$ that we deduced previously.
\end{proof}

\subsection{Proof of Theorem \ref{thm existence}, Case (iii)}
We first give the following lemma which will be useful for characterizing an upper bound for $m_c$.
\begin{lemma}\label{lem 2.5}
Let $\beta<0$ and $\alpha\in(0,\frac{8}{d+n})$. Then the following statements hold true:
\begin{itemize}
\item[(i)]We have
\begin{align}\label{2.20}
\|\Delta_{x,y}u\|_{2}^2+\frac{\beta^2}{4}\|u\|_2^2\geq-\beta\|\nabla_{x,y} u\|_2^2
\end{align}
for all $u\in H_{x,y}^2$.

\item[(ii)]For any $c\in(0,\infty)$ we have
\begin{align}\label{2.21}
\inf_{u\in S(c)}(\|\Delta_{x,y}u\|_{2}^2+\beta\|\nabla_{x,y} u\|_2^2)= -\frac{\beta^2c}{4}.
\end{align}

\item[(iii)]There exists some $c_0>0$ such that $c^{-1}m_c<-\beta^2/8$ for all $c\in(c_0,\infty)$.
\end{itemize}
\end{lemma}

\begin{proof}
Writing $u(x,y)=\sum_{k\in\Z^n}u_k(x)e^{ik\cdot y}$ as a Fourier series and using H\"older and Cauchy-Schwarz, we obtain
\begin{align}
-\beta\|\nabla_{x,y} u\|_2^2&=(2\pi)^n(-\beta)\sum_{k\in\Z^n}\int_{\R^d}|k|^2|\xi|^2|\hat{u}_k(\xi)|^2\,d\xi\nonumber\\
&\leq (2\pi)^n(-\beta)\bg(\sum_{k\in\Z^n}\int_{\R^d}|k|^4|\xi|^4|\hat{u}_k(\xi)|^2\,d\xi\bg)^{\frac12}
\bg(\sum_{k\in\Z^n}\int_{\R^d}|\hat{u}_k(\xi)|^2\,d\xi\bg)^{\frac12}\nonumber\\
&=2\|\Delta_{x,y} u\|_2\cdot\frac{(-\beta)\|\Delta_{x,y}u\|_{2}}{2}\leq \|\Delta_{x,y}u\|_{2}^2+\frac{\beta^2}{4}\|u\|_2^2
\end{align}
and \eqref{2.20} follows. Using \eqref{2.20} we already have $ \eqref{2.21}\geq -\beta^2c/4$. On the other hand, it is proved in \cite{Remark2019} that
\begin{align}\label{upper bound}
\inf_{u\in\widehat{S}(c)}(\|\Delta_{x}u\|_{L_x^2}^2+\beta\|\nabla_x u\|_{L_x^2}^2)=-\frac{\beta^2c}{4}.
\end{align}
By assuming that a candidate $u\in S(c)$ is independent of $y$ and using \eqref{upper bound} we obtain
$$ \eqref{2.21}\leq (2\pi)^n\inf_{u\in S((2\pi)^{-n}c)}(\|\Delta_{x,y}u\|_{2}^2+\beta\|\nabla_{x,y} u\|_2^2)=-\frac{\beta^2c}{4} $$
and (ii) follows. To prove (iii), we pick some $u\in S(1)$ and let $\tilde{u}:=\sqrt{c}u$. Then $M(\tilde{u})=c$ and consequently
\begin{align}\label{2.24}
c^{-1}m_{c}\leq c^{-1}E(\tilde{u})=\frac{1}{2}\|\Delta_{x,y}u\|_2^2+\frac{\beta}{2}\|\nabla_{x,y}u\|_2^2-\frac{c^{\frac{\alpha}{2}}}{\alpha+2}
\|u\|_{\alpha+2}^{\alpha+2}\to-\infty
\end{align}
as $c\to\infty$ and (iii) follows.
\end{proof}

\begin{proof}[Proof of Theorem \ref{thm existence}, Case (iii)]
By the proof of Theorem \ref{thm existence}, Case (i), it suffices to show that if $(u_n)_n$ is a minimizing sequence of $m_c$, then $\liminf_{n\to\infty}\|u_n\|_{\alpha+2}>0$. By the definition of $m_c$ and $c_-$ and Lemma \ref{lem 2.5} we have
\begin{align}
\liminf_{n\to\infty}\|u_n\|_{\alpha+2}^{\alpha+2}&=c(\alpha+2)(-c^{-1}m_c+(\|\Delta_{x,y}((\sqrt{c})^{-1}u_n)\|_{2}^2
+\beta\|\nabla_{x,y}((\sqrt{c})^{-1} u_n)\|_2^2)/2)\nonumber\\
&\geq c(\alpha+2)(-c^{-1}m_c-\beta^2/8)>0\label{2.25}
\end{align}
for $c>c_-$. That $m_c$ has no minimizer on $(0,c_-)$ can be similarly proved as in the proof of Theorem \ref{thm existence}, Case (ii) by using \eqref{2.19}. Finally, by \cite{Remark2019} we know that $c^{-1}\widehat{m}_c<-\beta^2/8$ for all $c\in(0,\infty)$ when $\alpha<\max\{\frac{8}{d+1},2\}$ and the the existence of a minimizer of $m_c$ for all $c\in(0,\infty)$ follows immediately from the previous arguments.
\end{proof}

\section{Proof of Theorem \ref{thm dependence}}

\subsection{Proof of Theorem \ref{thm dependence}, Case (i)}
Following the scaling arguments given by \cite{TTVproduct2014} we give the proof of Theorem \ref{thm dependence} (i) in this subsection.
\begin{lemma}\label{lem 3.1}
Let $u_\ld$ be an optimizer of $m_{1,\ld}$ (whose existence can be similarly deduced using the proof of Theorem \ref{thm existence}). Then we have
\begin{align}
\lim_{\ld\to\infty}m_{1,\ld}=(2\pi)^{n}\widehat{m}_{(2\pi)^{-n}}\label{3.3}
\end{align}
and
\begin{align}
\lim_{\ld\to\infty}\ld \|\Delta_y u_{\ld}\|_2^2=0.\label{3.4}
\end{align}
\end{lemma}

\begin{proof}
By assuming a candidate $u\in S(1)$ is independent of $y$ we already infer that $m_{1,\ld}\leq (2\pi)^n \widehat{m}_{(2\pi)^{-n}}$. Next, we prove that
\begin{align}
\lim_{\ld\to\infty}\|\Delta_y u_{\ld}\|_2^2=0.\label{3.5}
\end{align}
Suppose that this is not the case, then up to a subsequence we may assume that
\begin{align}
\inf_{\ld>0}\|\Delta_y u_{\ld}\|_2^2=\zeta>0.
\end{align}
Hence
\begin{align}
\lim_{\ld\to\infty}(\ld-1) \|\Delta_y u_{\ld}\|_2^2=\infty.
\end{align}
Let $\mu:=\frac{\alpha(d+n)}{4}$. Since $\alpha<\frac{8}{d+n}$, we know that $\mu<2$. Using \eqref{mc larger than minus inf} we deduce that there exist some $C_1,C_2>0$ such that
\begin{align} \label{3.6}
E_1(u)+\frac{M(u)}{2}\geq \inf_{t>0}(C_1t^2-C_2 t^\mu)=:C(\mu)>-\infty.
\end{align}
This in turn implies
\begin{align}
m_{1,\ld}-\frac{\ld-1}{2}\|\Delta_y u_\ld\|_2^2
=E_{\ld}(u_\ld)-\frac{\ld-1}{2}\|\Delta_y u_\ld\|_2^2
=\bg(E_1(u_\ld)+\frac{M(u_\ld)}{2}\bg)-\frac{M(u_\ld)}{2}\geq C(\mu)-\frac{1}{2}.
\end{align}
Now taking $\ld\to\infty$ we infer the contradiction $(2\pi)^n \widehat{m}_{(2\pi)^{-n}}\geq m_{1,\ld}\to \infty$. Thus \eqref{3.5} is proved. By interpolation we also have
\begin{align}
\lim_{\ld\to\infty}\|\nabla_y u_\ld\|_2^2=0.
\end{align}
Next, observe that estimating similarly as in \eqref{mc larger than minus inf}, $\|u_\ld\|_{H_{x,y}^2}\to \infty$ as $\ld\to\infty$ would in turn imply the contradiction $m_{1,\ld}\to\infty$, thus $(u_\ld)_\ld$ is a bounded sequence in $H_{x,y}^2$. We also point out that for $v\in \dot{H}_y^2$ we have
$$ \|v\|_{\dot{H}_y^2}\lesssim \|\Delta_y v\|_{L_y^2}.$$
To see this, writing $v(y)=\sum_{k\in \Z^n}v_k e^{ik\cdot y}$ and using Cauchy-Schwarz we have
\begin{align}
\sum_{i,j=1}^n\|\pt_{y_i}\pt_{y_j}v\|_{L_y^2}^2&=(2\pi)^n
\sum_{i,j=1}^n\sum_{k\in\Z^n}|k_i k_j|^2|v_k|^2\lesssim
\sum_{i,j=1}^n(\sum_{k\in\Z^n}k_i^4|v_k|^2)^{\frac12}(\sum_{k\in\Z^n}k_j^4|v_k|^2)^{\frac12}\nonumber\\
&\leq \sum_{i,j=1}^n(\sum_{i=1}^n\sum_{k\in\Z^n}k_i^4|v_k|^2)^{\frac12}(\sum_{j=1}^n\sum_{k\in\Z^n}k_j^4|v_k|^2)^{\frac12}
=2n\|\Delta_y  v\|_{L_y^2}^2.
\end{align}
Now, define
\begin{align}
w_\ld(y):=\|u_\ld(\cdot,y)\|_{L_x^2}^2.
\end{align}
Then $\|w_\ld\|_{L_y^1}=1$. Using product rule, H\"older and \eqref{3.5} we also infer that
\begin{align}
\|\nabla_y w_\ld\|_{L_y^1}&\lesssim \|u_\ld\|_2\|\nabla_y u_\ld\|_2\to 0,\label{3.12}\\
\|\nabla^2_{y}w_\ld\|_{L_y^1}&\lesssim \|\nabla_y u_\ld\|_2^2+\|u_\ld\|_2\|\Delta_y u_\ld\|_2^2\to 0\label{3.13}
\end{align}
as $\ld\to\infty$. Hence $(w_\ld)_\ld$ is a bounded sequence in $W^{2,1}(\T^n)$. By fundamental rescaling arguments it is easy to show that $\widehat{m}_c=c^{\frac{d-8/\alpha-4}{d-8/\alpha}}\widehat{m}_1$. Therefore
\begin{align}
m_{1,\ld}&=E_{\ld}(u_\ld)\geq \int_{y\in\T^n}\widehat{E}(u_\ld(\cdot,y))\,dy
\geq\int_{y\in\T^n}\widehat{m}_{\|u_\ld(\cdot,y)\|^2_{L_x^2}}\,dy\nonumber\\
&=\int_{y\in\T^n}\widehat{m}_{1}(\|u_\ld(\cdot,y)\|^2_{L_x^2})^{\frac{d-8/\alpha-4}{d-8/\alpha}}\,dy
=\widehat{m}_{1}\|w_\ld\|_{L_y^{\frac{d-8/\alpha-4}{d-8/\alpha}}}^{\frac{d-8/\alpha-4}{d-8/\alpha}}.\label{3.14}
\end{align}
Using $\alpha\in (0,\frac{8}{d+n})$ we deduce that $\frac{d-8/\alpha-4}{d-8/\alpha}\in(1,1+\frac{4}{n})$. When $n\leq 4$, by Rellich's compact embedding we have $W_y^{2,1}\hookrightarrow \hookrightarrow L_y^{r}$ for all $r\in[1,1+4/n)$. Using \eqref{3.12} and $\|w_\ld\|_{L_y^1}\equiv 1$ we then obtain
\begin{align}
\lim_{\ld\to\infty}\|w_\ld\|_{L_y^{\frac{d-8/\alpha-4}{d-8/\alpha}}}^{\frac{d-8/\alpha-4}{d-8/\alpha}}
=(2\pi)^n((2\pi)^n)^{\frac{4+8/\alpha-d}{d-8/\alpha}}.\label{3.15}
\end{align}
For $n\geq 5$, we use Minkowski's inequality and the Sobolev embedding $H^2_y\hookrightarrow W_y^{1,\frac{2n}{n-2}}\hookrightarrow L_y^{\frac{2n}{n-4}}$ to obtain
\begin{align}
\sup_{\ld>0}\|w_\ld\|_{L_y^{\frac{n}{n-4}}}= \sup_{\ld>0}\|u_\ld\|^2_{L_y^{\frac{2n}{n-4}}L_x^2}
\lesssim \sup_{\ld>0}\|u_\ld\|^2_{L_x^2L_y^{\frac{2n}{n-4}}}\lesssim
\sup_{\ld>0}\|u_\ld\|^2_{L_x^2H_y^2}<\infty.
\end{align}
Combining with $W_y^{2,1}\hookrightarrow \hookrightarrow L_y^{1}$ and interpolation we arrive at \eqref{3.15} again. Thus using \eqref{3.14}
\begin{align}
\liminf_{\ld\to\infty} m_{1,\ld}\geq (2\pi)^n(\widehat{m}_1((2\pi)^n)^{\frac{4+8/\alpha-d}{d-8/\alpha}})=(2\pi)^n\widehat{m}_{(2\pi)^{-n}},\label{3.17}
\end{align}
where the last inequality is easily deduced using rescaling. \eqref{3.3} follows now from \eqref{3.17} and the fact that $m_{1,\ld}\leq (2\pi)^n\widehat{m}_{(2\pi)^{-n}}$ for all $\ld>0$. Finally, by keeping $\frac{\ld}{2}(\|\Delta_y u_\ld\|_2^2+\|\pt_y u_\ld\|_2^2)$ in \eqref{3.14} we have
\begin{align}
m_{1,\ld}\geq \frac{\ld}{2}(\|\Delta_y u_\ld\|_2^2+\|\pt_y u_\ld\|_2^2)+\widehat{m}_{1}\|w_\ld\|_{L_y^{\frac{d-8/\alpha-4}{d-8/\alpha}}}^{\frac{d-8/\alpha-4}{d-8/\alpha}}.\label{3.18}
\end{align}
\eqref{3.4} follows now from \eqref{3.3}, \eqref{3.15} and taking $\ld$ in \eqref{3.18} to infinity.
\end{proof}

\begin{lemma}\label{lem 4.2}
For any $\ld>0$ we have
\begin{align}
2\|\Delta_x u_\ld\|_2^2-\frac{\alpha d}{2(\alpha+2)}\|u_\ld\|_{\alpha+2}^{\alpha+2}=0.\label{pohozaev}
\end{align}
Consequently, there exists some $\theta_\ld\in\R$ such that
\begin{align}
\Delta^2_{x}u_\ld+\ld\Delta_y^2 u_\ld+\theta_\ld u_\ld=|u|^\alpha u.\label{lagrange}
\end{align}
Moreover, we have $\lim_{\ld\to\infty}\theta_\ld=\bar{\theta}\in(0,\infty)$. In fact,
\begin{align}\label{theta bar}
\bar{\theta}=-\frac{2(\alpha d-4\alpha-8)}{\alpha d-8}((2\pi)^n \widehat{m}_{(2\pi)^{-n}}).
\end{align}
\end{lemma}

\begin{proof}
\eqref{pohozaev} follows from $\frac{d}{dt}E_\ld(u_\ld^t)|_{t=1}=0$, where $u^t(x,y):=t^{\frac{d}{2}}u(tx,y)$. The fact that $\frac{d}{dt}E_\ld(u_\ld^t)|_{t=1}=0$ follows from $u_\ld$ being a minimizer of $m_{1,\ld}$. \eqref{lagrange} is a direct consequence of Lagrange multiplier theorem. Next, testing \eqref{lagrange} with $\bar{u}_\ld$ and combining with \eqref{pohozaev}, \eqref{3.3}, \eqref{3.4} and $M(u_\ld)=1$ we obtain
\begin{align}
\theta_\ld=-\frac{2(\alpha d-4\alpha-8)}{\alpha d-8}((2\pi)^n \widehat{m}_{(2\pi)^{-n}})+o_\ld(1)
\end{align}
and the desired claim follows.
\end{proof}

\begin{lemma}
\label{lem 3.3}  There exists some $v\in \widehat{S}((2\pi)^{-n})$ such that, up to subsequence and $\R^d$-translations, $u_\ld$ converges strongly in $H_{x,y}^2$ to $v$ as $\ld\to\infty$, $\widehat{E}(v)=\widehat{m}_{(2\pi)^{-n}}$ and $v$ satisfies
\begin{align}
\Delta_x^2 v+\bar{\theta}v=|v|^\alpha v.\label{lagrange v}
\end{align}
\end{lemma}

\begin{proof}
From the proof of Lemma \ref{lem 3.1} we know that $(u_\ld)_\ld$ is a bounded sequence in $H_{x,y}^2$. Moreover, using \eqref{pohozaev}, \eqref{3.3} and \eqref{3.4} we also have
\begin{align}
\|u_\ld\|_{\alpha+2}^{\alpha+2}=\frac{8(\alpha+2)}{\alpha d-8}((2\pi)^n\widehat{m}_{(2\pi)^{-n}})+o_\ld(1).
\end{align}
Thus $\liminf_{\ld\to\infty}\|u_\ld\|_{\alpha+2}>0$ and by Lemma \ref{lem:non-vanishing weak limit}, up to subsequence and $\R^d$-translations, $u_\ld$ converges weakly to some $v\in H_{x,y}^2\setminus\{0\}$ in $H_{x,y}^2$. Moreover, the weak convergence of $u_\ld$ to $v$ in $H_{x,y}^2$, the Sobolev embedding, \eqref{3.4}, \eqref{lagrange} and $\lim_{\ld\to\infty}\theta_\ld=\bar{\theta}$ also yield \eqref{lagrange v}. It is left to show that $v$ is a minimizer of $\widehat{m}_{(2\pi)^{-n}}$. By weakly lower semicontinuity of norm we have $\widehat{M}(v)\in(0,(2\pi)^{-n}]$. We then firstly show that $\widehat{M}(v)=(2\pi)^{-n}$. Assume that $\widehat{M}(v)=c_1\in (0,(2\pi)^{-n})$. Using Br\'{e}zis-Lieb lemma, \eqref{3.3} and \eqref{3.4} we have
\begin{align}
E_0(u_\ld-v)+E_0(v)&=(2\pi)^n\widehat{m}_{(2\pi)^{-n}}+o_\ld(1),\\
M(u_\ld-v)+(2\pi)^nc_1&=1+o_\ld(1).\label{mass decomp}
\end{align}
Define $z_\ld:=w_\ld-v$. Then $(z_\ld)_\ld$ is a bounded sequence in $W_y^{2,1}$. More precisely, using \eqref{mass decomp}, \eqref{3.12} and \eqref{3.13} we have
\begin{gather}
\|z_\ld\|_{L_y^1}=1-(2\pi)^nc_1+o_\ld(1),\\
\|\nabla_y z_\ld\|_{L_y^1}+\|\nabla_y^2 z_\ld\|_{L_y^1}=o_\ld(1).
\end{gather}
Then arguing as in \eqref{3.17} we infer that
$$ \liminf_{\ld\to\infty} E_0(u_\ld-v)\geq (2\pi)^n\widehat{m}_{(2\pi)^{-n}-c_1}.$$
By definition we also know that $E_0(v)\geq (2\pi)^n\widehat{m}_{c_1}$. Arguing as in Step 3 of the proof of Theorem \ref{thm existence}, Case (i) we deduce that the mapping $c\mapsto c^{-1}\widehat{m}_{c}$ is strictly decreasing on $(0,\infty)$. Summing up we conclude that
\begin{align}
(2\pi)^n\widehat{m}_{(2\pi)^{-n}}\geq (2\pi)^n\widehat{m}_{c_1}+(2\pi)^n\widehat{m}_{(2\pi)^{-n}-c_1}
> (2\pi)^{n}\widehat{m}_{(2\pi)^{-n}},\label{contra}
\end{align}
a contradiction. Thus $\widehat{M}(v)=(2\pi)^{-n}$. Moreover, if $\widehat{E}(v)>\widehat{m}_{(2\pi)^{-n}}$, then repeating the previous calculation, we arrive at the contradiction \eqref{contra} again. Thus $v$ is a minimizer of $\widehat{m}_{(2\pi)^{-n}}$. Finally, since $v$ minimizes $\widehat{m}_{(2\pi)^{-n}}$, we see that all the inequalities given previously are in fact equalities. Thus weak convergence and convergence in norms imply the strong convergence of $u_\ld$ to $v$. This completes the desired proof.
\end{proof}

\begin{lemma}\label{lemma no dependence}
There exists some $\ld_0>0$ such that for all $\ld>\ld_0$ we have $\nabla_y u_\ld=0$.
\end{lemma}

\begin{proof}
Define $w_\ld:=\sqrt{-\Delta_y}u_\ld$. Then applying $-\Delta_y$ to \eqref{lagrange} we obtain
\begin{align}
\Delta_x^2 w_\ld+\ld \Delta_y^2 w_\ld+\theta_\ld w_\ld=\sqrt{-\Delta_y}(|u_\ld|^\alpha u_\ld).\label{3.26}
\end{align}
For $z\in\C$ let $F(z):=|z|^{\alpha}z$. By writing $z=a+bi$ with $a,b\in\R$ we may identify $F$ as a complex function from $\R^2$ to $\C$. Then we define the standard complex derivatives
\begin{align*}
F_z:=\frac{1}{2}\bg(\frac{\pt F}{\pt a}-i\frac{\pt F}{\pt b}\bg),\quad F_{\bar{z}}:=\frac{1}{2}\bg(\frac{\pt F}{\pt a}+i\frac{\pt F}{\pt b}\bg).
\end{align*}
Using the chain rule, we infer that for $u:\T^n\to \C$ we have
\begin{align}
\nabla_y F(u(y))=F_z(u(y))\nabla_y u(y)+F_{\bar{z}}(u(y))\overline{\nabla_y u(y)}.
\end{align}
Now define $G_\ld(x,y):=F_z(v(x))u_\ld(x,y)+F_{\bar{z}}(v(x))\overline{u_\ld(x,y)}$. Then testing \eqref{3.26} with $\bar{w}_\ld$ yields
\begin{align}
0=\|\Delta_x w_\ld\|_2^2+\ld\|\Delta_y w_\ld\|_2^2+\theta_\ld M(w_\ld)-\int_{\R^d\times\T^n}
\sqrt{-\Delta_y}(|u_\ld|^\alpha u_\ld)w_\ld\,dxdy
\end{align}
or equivalently
\begin{align}
0&=(\ld-1)\|\Delta_y w_\ld\|_2^2-\int_{\R^d\times\T^n}
(\sqrt{-\Delta_y}G_\ld)w_\ld\,dxdy\nonumber\\
&+\|\Delta_{x,y}w_\ld\|_2^2+\bar{\theta}M(w_\ld)+\int_{\R^d\times\T^n}\sqrt{-\Delta_y}(G_\ld-|u_\ld|^\alpha u_\ld)w_\ld\,dxdy\nonumber\\
&+(\theta_\ld-\bar{\theta})M(w_\ld)=:I+II+III.
\end{align}
Using standard elliptic regularity theory and bootstrap arguments (see for instance the proof of \cite[Thm. 8.1.1]{Cazenave2003}) and Sobolev embedding we are able to show that $v\in L_x^\infty$. For $I$, writing $w_\ld(x,y)=(2\pi)^{-n}\sum_{k\in\Z^n}a_{j,k}(x)e^{ik\cdot y}$ (that $a_{j,0}=0$ follows from $\int_{\T^n}w_\ld(x,y)\,dy=0$) and using $v\in L_x^\infty$ we infer that
\begin{align}
I\geq (\ld-1)\sum_{k\in\Z^n\setminus\{0\}}|k|^4\int_{\R}|a_{j,k}(x)|^2\,dx-C(\|v\|_{L_x^\infty},\alpha)\sum_{k\in\Z^n\setminus\{0\}}\int_{\R}|a_{j,k}(x)|^2\,dx\geq 0
\end{align}
for $\ld\gg 1$. Moreover, the terms $II$ and $III$ are exactly the terms $II_j$ and $III_j$ respectively in the proof of \cite[Lem. 3.6]{TTVproduct2014}, and from therein we conclude that
$$ II+III\gtrsim \|w_\ld\|_{H_{x,y}^1}^2(1-o_\ld(1)).$$
Thus the desired claim follows by taking $\ld $ sufficiently large.
\end{proof}

\begin{lemma}\label{lem 3.5}
We have $\lim_{\ld\to0}m_{1,\ld}<(2\pi)^n\widehat{m}_{(2\pi)^{-n}}$.
\end{lemma}

\begin{proof}
We firstly define the function $\rho:[0,2\pi]\to[0,\infty)$ as follows: Let $a\in(0,\pi)$ and $b\in(0,\infty)$ be some to be determined positive numbers. Then we define $\rho$ by
\begin{align*}
\rho(z)=\left\{
\begin{array}{ll}
0,&z\in[0,a]\cup[2\pi-a,2\pi],\\
b(z-a),&z\in[a,\pi],\\
\rho(2\pi-z),&z\in[\pi,2\pi].
\end{array}
\right.
\end{align*}
Direct calculation shows that
\begin{align}
\|\rho\|_{L_y^2}^2=\frac{2b^2(\pi-a)^3}{3}\quad\text{and}\quad\|\rho\|_{L^{\alpha+2}_y}^{\alpha+2}=\frac{2b^{\alpha+2}(\pi-a)^{\alpha+3}}{\alpha+3}.
\end{align}
We then fix the value of $b$ by
$$ b=\bg[\frac{(\alpha+3)(\pi-a)^{-(\alpha+1)}}{3}\bg]^{\frac{1}{\alpha}}.$$
This choice of $b$ also yields $\|\rho\|_{L_y^2}^2=\|\rho\|_{L^{\alpha+2}_y}^{\alpha+2}$. Moreover,
$$ \|\rho\|_{L_y^2}^2=\frac{2}{3}\bg(\frac{\alpha+3}{3}\bg)^{\frac{2}{\alpha}}(\pi-a)^{1+\frac{2}{\alpha}}.$$
Let $\rho_\vare$ be the $\vare$-mollifier of $\rho$ on $[0,2\pi]$ with some to be determined small $\vare>0$. In particular, since $\rho$ has compact support in $(0,2\pi)$, so is $\rho_\vare$ for $\vare\ll 1$. Next, let $Q$ be an optimizer of $\widehat{m}_{\|\rho\|_{L_y^2}^{-2n}}$ and define
\begin{align}
\psi(x,y):=Q(x)(\|\rho\|_{L_y^2}/\|\rho_\vare\|_{L_y^2})^{n}\prod_{j=1}^n\rho_\vare(y_j).
\end{align}
This is to be understood that we extend $\rho_\vare$ $2\pi$-periodically along the $y$-direction, which is possible since $\rho_\vare$ has compact support in $(0,2\pi)$ when $\vare\ll 1$. We have then $M(\psi)=1$. Moreover,
\begin{align}
E_0(\psi)&=\frac{1}{2}\|\rho\|_{L_y^2}^{2n}\|\Delta_x Q\|_{L_x^2}^2
-\frac{1}{\alpha+2}\|\rho\|_{L_y^2}^{(\alpha+2)n}
\|\rho_\vare\|_{L_y^{2}}^{-(\alpha+2)n}\|\rho_\vare\|_{L_y^{\alpha+2}}^{(\alpha+2)n}\|Q\|_{L_x^{\alpha+2}}^{\alpha+2}\nonumber\\
&=\|\rho\|_{L_y^2}^{2n}\bg(\frac{1}{2}\|\Delta_x Q\|_{L_x^2}^2
-\frac{1}{\alpha+2}\|Q\|_{L_x^{\alpha+2}}^{\alpha+2}\bg)\nonumber\\
&-\frac{\|\rho\|_{L_y^2}^{2n}}{\alpha+2}
\bg(\|\rho\|_{L_y^2}^{-2n}\|\rho\|_{L_y^2}^{(\alpha+2)n}\|\rho_\vare\|_{L_y^{2}}^{-(\alpha+2)n}\|\rho_\vare\|_{L_y^{\alpha+2}}^{(\alpha+2)n}-1\bg)
\|Q\|_{L_x^{\alpha+2}}^{\alpha+2}\nonumber\\
&=:\|\rho\|_{L_y^2}^{2n}\widehat{m}_{\|\rho\|_{L_y^2}^{-2n}}+I.
\end{align}
Recall from Step 3 of the proof of Theorem \ref{thm existence}, Case (i) that the mapping $c\mapsto c^{-1}\widehat{m}_c$ is strictly decreasing on $(0,\infty)$. Noticing also that $\|\rho\|_{L_y^2}\to 0$ as $a\to\pi$. Then by choosing $a$ sufficiently close to $\pi$ we conclude that
$$\|\rho\|_{L_y^2}^{2n}\widehat{m}_{\|\rho\|_{L_y^2}^{-2n}}<(2\pi)^{n}\widehat{m}_{(2\pi)^{-n}}.$$
Thus we can find some $\zeta>0$ such that $\|\rho\|_{L_y^2}^{2n}\widehat{m}_{\|\rho\|_{L_y^2}^{-2n}}+\zeta<(2\pi)^{n}\widehat{m}_{(2\pi)^{-n}}$. By the properties of a mollifier operator we also know that
\begin{align}
\|\rho_\vare\|^{2}_{L_y^2}=\|\rho\|_{L_y^2}^2+o_\vare(1)=\|\rho\|_{L_y^{\alpha+2}}^{\alpha+2}+o_\vare(1)
=\|\rho_\vare\|_{L_y^{\alpha+2}}^{\alpha+2}+o_\vare(1).
\end{align}
Hence taking $\vare$ sufficiently small we conclude that $|I|\leq \zeta$ and summing up we deduce $E_0(\psi)<(2\pi)^{n}\widehat{m}_{(2\pi)^{-n}}$ by choosing $a,b,\vare$ as above. Consequently,
\begin{align}
\lim_{\ld\to 0}m_{1,\ld}\leq \lim_{\ld\to 0}\mH_{\ld}(\psi)=\mH_0(\psi)<2\pi\wm_{(2\pi)^{-1}},
\end{align}
as desired.
\end{proof}

\begin{lemma}\label{lemma auxiliary}
There exists some $\ld_*\in(0,\infty)$ such that
\begin{itemize}
\item For all $\ld \in (0,\ld_*)$ we have $m_{1,\ld}<(2\pi)^n \wm_{(2\pi)^{-n}}$. Moreover, for $\ld\in(0,\ld_*)$ any minimizer $u_\ld$ of $m_{1,\ld}$ satisfies $\pt_y u_\ld\neq 0$.

\item For all $\ld\in(\ld_*,\infty)$ we have $m_{1,\ld}=(2\pi)^n \wm_{(2\pi)^{-n}}$. Moreover, for $\ld\in(\ld_*,\infty)$ any minimizer $u_\ld$ of $m_{1,\ld}$ satisfies $\pt_y u_\ld=0$.
\end{itemize}
\end{lemma}

\begin{proof}
Define
\begin{align*}
\ld_*:=\inf\{\tilde{\ld}\in(0,\infty):m_{1,\ld}=(2\pi)^n \wm_{(2\pi)^{-n}}\,\forall\,\ld\geq \tilde{\ld}\}.
\end{align*}
From Lemma \ref{lemma no dependence} and Lemma \ref{lem 3.5} we know that $\ld_*\in(0,\infty)$. The fact that a minimizer of $m_{1,\ld}$ has non-trivial $y$-dependence and $m_{1,\ld}<(2\pi)^n\widehat{m}_{(2\pi)^{-n}}$ for $\ld<\ld_*$ follows already from the definition of $\ld_*$ and the fact that $\ld\mapsto m_{1,\ld}$ is monotone increasing. We borrow an idea from \cite{GrossPitaevskiR1T1} to show that any minimizer of $m_{1,\ld}$ for $\ld>\ld_*$ must be $y$-independent: Assume the contrary that an optimizer $u_\ld$ of $m_{1,\ld}$ satisfies $\|\pt_y u_\ld\|_2^2\neq 0$. Then there exists some  $\mu\in(\ld_*,\ld)$. Consequently,
\begin{align*}
(2\pi)^n \wm_{(2\pi)^{-n}}=m_{1,\mu}\leq \mH_{\mu}(u_\ld)=\mH_{\ld}(u_\ld)+\frac{\mu-\ld}{2}\|\pt_y u_\ld\|_2^2<\mH_{\ld}(u_\ld)=m_{1,\ld}=2\pi \wm_{(2\pi)^{-1}},
\end{align*}
a contradiction. This completes the desired proof.
\end{proof}

\begin{proof}[Proof of Theorem \ref{thm dependence}, Case (i)]
For $c>0$ let $\kappa_c:=c^{\frac{1}{d-\frac8\alpha}}$. For $\theta>0$ define the operator $T_\theta u$ by
$$T_\theta u(x,y):=\theta^{\frac{4}{\alpha}}u(\theta x,y).$$
Then $u\mapsto T_{\kappa_c} u$ defines a bijection between $S(c)$ and $S(1)$. Direct calculation also shows that $\mH(u)=c^{\frac{d-8/\alpha-4}{d-8/\alpha}}{\mH_{\kappa_c^{4}}}(T_{\kappa_c} u)$, thus $m_{c}=c^{\frac{d-8/\alpha-4}{d-8/\alpha}}{m_{1,\kappa_c^{4}}}$. By same rescaling arguments we also infer that $\wm_{(2\pi)^{-n}c}=c^{\frac{d-8/\alpha-4}{d-8/\alpha}}\wm_{(2\pi)^{-n}}$ for $c>0$. Notice also that the mapping $c\mapsto \kappa_c$ is strictly monotone decreasing on $(0,\infty)$. Thus by Lemma \ref{lemma auxiliary} we know that there exists some $c_0\in(0,\infty)$ such that
\begin{itemize}
\item For all $c\in(0,c_0)$ we have
$$m_{c}=c^{\frac{d-8/\alpha-4}{d-8/\alpha}}{m_{1,\kappa_c^{4}}}=c^{\frac{d-8/\alpha-4}{d-8/\alpha}}(2\pi)^n \wm_{(2\pi)^{-n}}=(2\pi)^n \wm_{(2\pi)^{-n}c}.$$

\item For all $c\in(c_0,\infty)$ we have
$$m_{c}=c^{\frac{d-8/\alpha-4}{d-8/\alpha}}{m_{1,\kappa_c^{4}}}<c^{\frac{d-8/\alpha-4}{d-8/\alpha}}(2\pi)^n \wm_{(2\pi)^{-n}}=(2\pi)^n \wm_{(2\pi)^{-n}c}.$$
\end{itemize}
The statements concerning the $y$-dependence of the minimizers follow also from Lemma \ref{lemma auxiliary} simultaneously. Finally, that $m_{c_0}=2\pi \wm_{(2\pi)^{-1}c_0}$ follows immediately from the continuity of the mappings $c\mapsto m_c$ and $c\mapsto \wm_c$. This completes the proof.
\end{proof}

\subsection{Proof of Theorem \ref{thm dependence}, Case (ii)}

As already mentioned in the introductory section, in the case $\beta\neq 0$ we encounter the new difficulty that \eqref{swnls} is no longer scale-invariant.
%We will overcome this difficulty by using a new scaling argument to show $y$-independence of the ground states with small mass when $\beta > 0$, and by constructing some suitable test functions in order to establish $y$-dependence of the ground states with sufficiently large mass.
When $\beta>0$ and $\alpha\in(0,\frac{4}{d+n})$, we partially solve this problem by using a new scaling argument and prove $y$-independence of the ground states with small mass. The crucial scaling lemma, which plays an analogous role as Lemma \ref{lem 3.1}, is given as follows.

\begin{lemma}\label{lem C.1}
	Let {$\beta>0$, $\alpha\in(0,\frac{4}{d+n})$ and} $u_\tau$ be an optimizer of $\mu_{1,\tau}$ (whose existence can be similarly deduced using the proof of Theorem \ref{thm existence}). Then we have
	\begin{align}
		\lim_{\tau\to\infty}\mu_{1,\tau}=(2\pi)^n\widehat{\mu}_{(2\pi)^{-n}}^{\infty}\label{C.3}
	\end{align}
	and
	\begin{align}
		\lim_{\tau\to\infty}\tau \|\Delta_y u_{\tau}\|_2^2=0,\label{C.4} \\
		\lim_{\tau\to\infty}\tau \|\nabla_y u_{\tau}\|_2^2=0,\label{C.19} \\
		\lim_{\tau\to\infty}\frac{1}{\tau} \|\Delta_x u_{\tau}\|_2^2=0.\label{C.20}
	\end{align}
\end{lemma}

\begin{proof}
	Let $u=u(x) \in H^2_x$ be such that $\widehat{M}(u) = (2\pi)^{-n}$ and $\widehat{E}^{\infty}(u) = \widehat{\mu}_{(2\pi)^{-n}}^{\infty}$ ({{the existence of such a function $u$ is guaranteed by \eqref{equalitya} and its following discussion}}). Then $u \in H^2_{x,y}$ and $M(u) = 1$. Thus we get
	\begin{align}
		\mu_{1,\tau} \leq \mH^{\tau}(u) = (2\pi)^n\left(\frac{1}{2\tau}\|\Delta_{x} u\|^2_{L_x^2}+\frac{\beta}{2}\|\nabla_{x} u\|^2_{L_x^2} -\frac{1}{\alpha+2}\|u\|^{\alpha+2}_{L_x^{\alpha+2}}\right) = (2\pi)^n\widehat{\mu}_{(2\pi)^{-n}}^{\infty} + o_\tau(1),
	\end{align}
	where $\lim_{\tau \to \infty}o_\tau(1) = 0$, implying that $\limsup_{\tau \to \infty}\mu_{1,\tau} \leq (2\pi)^n\widehat{\mu}_{(2\pi)^{-n}}^{\infty}$.
	
	Noth that
	\begin{align}
	\mu_{1,\tau} =& \mH^{\tau}(u_{\tau}) \nonumber \\
	=& \frac{\tau}{2}\|\Delta_{y} u_{\tau}\|^2_{2}+\frac{\beta(\tau-1)}{2}\|\nabla_{y} u_{\tau}\|^2_{2}+\frac{1}{2\tau}\|\Delta_{x} u_{\tau}\|^2_{2} \nonumber \\
	& + \frac{\beta}{2}\|\nabla_{y} u_{\tau}\|^2_{2} + \frac{\beta}{2}\|\nabla_{x} u_{\tau}\|^2_{2} -\frac{1}{\alpha+2}\|u_{\tau}\|^{\alpha+2}_{\alpha+2}.
	\end{align}
	Since $\alpha \in (0,\frac{4}{d+n})$, similar to the proof of \eqref{3.6}, by the classical Gagliardo-Nirenberg inequality, we obtain
	\begin{align}
		\inf_{\tau>0}\bg(\frac{\beta}{2}\|\nabla_{y} u_{\tau}\|^2_{2} + \frac{\beta}{2}\|\nabla_{x} u_{\tau}\|^2_{2} -\frac{1}{\alpha+2}\|u_{\tau}\|^{\alpha+2}_{\alpha+2} \bg)> -\infty.
	\end{align}
	Hence,
	\begin{align}
		0 \leq \limsup_{\tau \to \infty}\tau\|\Delta_{y} u_{\tau}\|^2_{2} <\infty, \\
		0 \leq\limsup_{\tau \to \infty}\,(\tau-1)\|\nabla_{y} u_{\tau}\|^2_{2} <\infty, \\
		0 \leq\limsup_{\tau \to \infty}\frac{1}{\tau}\|\Delta_{x} u_{\tau}\|^2_{2} <\infty,
	\end{align}
	showing that
	\begin{align}
		\lim_{\tau\to\infty}\|\Delta_{y} u_{\tau}\|^2_{2}=0, \\
		\lim_{\tau\to\infty}\|\nabla_{y} u_{\tau}\|^2_{2}=0.
	\end{align}
Next, observe that estimating similarly as in \eqref{mc larger than minus inf}, $\|u_\tau\|_{H_{x,y}^1}\to \infty$ as $\tau\to\infty$ would in turn imply the contradiction $\mu_{1,\tau}\to\infty$, thus $(u_\tau)_\tau$ is a bounded sequence in $H_{x,y}^1$. Then similar to the proof of Lemma \ref{lem 3.1}, defining
	\begin{align}
		w_\tau(y):=\|u_\tau(\cdot,y)\|_{L_x^2}^2
	\end{align}
	we conclude that $(w_\tau)_\tau$ is a bounded sequence in $W^{2,1}(\T^n)$ with $\|w_\tau\|_{L_y^1}=1$. Therefore
	\begin{align}
		\mu_{1,\tau}&=E^{\tau}(u_\tau)\geq \int_{y\in\T^n}\widehat{E}^{\infty}(u_\tau(\cdot,y))\,dy
		\geq\int_{y\in\T^n}\widehat{\mu}^{\infty}_{\|u_\tau(\cdot,y)\|^2_{L_x^2}}\,dy\nonumber\\
		&=\int_{y\in\T^n}\widehat{\mu}_{1}^{\infty}(\|u_\tau(\cdot,y)\|^2_{L_x^2})^{\frac{d-4/\alpha-2}{d-4/\alpha}}\,dy
		=\widehat{\mu}^\infty_{1}\|w_\tau\|_{L_y^{\frac{d-4/\alpha-2}{d-4/\alpha}}}^{\frac{d-4/\alpha-2}{d-4/\alpha}}.\label{C.14}
	\end{align}
	Using $\alpha\in (0,\frac{4}{d+n})$ we deduce that $\frac{d-4/\alpha-2}{d-4/\alpha}\in(1,1+\frac{2}{n})$. Similar to the proof of \eqref{3.15} we then obtain
	\begin{align}
		\lim_{\tau\to\infty}\|w_\tau\|_{L_y^{\frac{d-4/\alpha-2}{d-4/\alpha}}}^{\frac{d-4/\alpha-2}{d-4/\alpha}}
		=(2\pi)^n((2\pi)^n)^{\frac{2+4/\alpha-d}{d-4/\alpha}}.\label{C.15}
	\end{align}
    Thus using \eqref{C.14}
	\begin{align}
		\liminf_{\tau\to\infty} \mu_{1,\tau}\geq (2\pi)^n(\widehat{\mu}^{\infty}_1((2\pi)^n)^{\frac{2+4/\alpha-d}{d-4/\alpha}})=(2\pi)^n\widehat{\mu}^{\infty}_{(2\pi)^{-n}}.\label{C.17}
	\end{align}
	\eqref{C.3} follows now from \eqref{C.17} and the fact that $\limsup_{\tau \to \infty}\mu_{1,\tau} \leq (2\pi)^n\widehat{\mu}_{(2\pi)^{-n}}^{\infty}$. Finally, by keeping $\frac{\tau}{2}\|\Delta_y u_\tau\|_2^2 + \frac{\tau\beta}{2}\|\nabla_y u_\tau\|_2^2 + \frac{1}{2\tau}\|\Delta_x u_\tau\|_2^2$ in \eqref{C.14} we have
	\begin{align}
		\mu_{1,\tau}\geq \frac{\tau}{2}\|\Delta_y u_\tau\|_2^2 + \frac{\tau\beta}{2}\|\nabla_y u_\tau\|_2^2 + \frac{1}{2\tau}\|\Delta_x u_\tau\|_2^2+\widehat{\mu}^\infty_{1}\|w_\tau\|_{L_y^{\frac{d-4/\alpha-2}{d-4/\alpha}}}^{\frac{d-4/\alpha-2}{d-4/\alpha}}.\label{C.18}
	\end{align}
	\eqref{C.4}, \eqref{C.19} and \eqref{C.20} follow now from \eqref{C.3}, \eqref{C.15} and taking $\tau$ in \eqref{C.18} to infinity.
\end{proof}

By making use of Lemma \ref{lem C.1} in place of Lemma \ref{lem 3.1} we immediately deduce the following lemmas which play the same roles as Lemma \ref{lem 4.2} to \ref{lemma no dependence} given in last section.

\begin{lemma}\label{lem 4.8}
	{Let $\beta>0$, $\alpha\in(0,\frac{4}{d+n})$.} For any $\tau>0$ we have
	\begin{align}
		\frac{2}{\tau}\|\Delta_x u_\tau\|_2^2 + \beta\|\nabla_x u_\tau\|_2^2 -\frac{\alpha d}{2(\alpha+2)}\|u_\tau\|_{\alpha+2}^{\alpha+2}=0.\label{pohozaev 2}
	\end{align}
	Consequently, there exists some $\theta_\tau\in\R$ such that
	\begin{align}
		\tau\Delta_y^2 u_\tau + \tau\beta\Delta_y u_\tau + \frac{1}{\tau}\Delta^2_{x}u_\tau+\beta\Delta_x u_\tau+\theta_\tau u_\tau=|u_\tau|^\alpha u_\tau.\label{lagrange 2}
	\end{align}
	Moreover, we have $\lim_{\tau\to\infty}\theta_\tau=\hat{\theta}\in(0,\infty)$. In fact,
	\begin{align}\label{theta hat}
		\hat{\theta}=-\frac{2(\alpha d-2\alpha-4)}{\alpha d-4}((2\pi)^n \widehat{\mu}^{\infty}_{(2\pi)^{-n}}).
	\end{align}
\end{lemma}

\begin{proof}
%\eqref{pohozaev 2} follows from $\frac{d}{dt}E^\tau(u_\tau^t)|_{t=1}=0$, where $u^t(x,y):=t^{\frac{d}{2}}u(tx,y)$. The fact that $\frac{d}{dt}E^\tau(u_\tau^t)|_{t=1}=0$ follows from $u_\tau$ being a minimizer of $\mu_{1,\tau}$. \eqref{lagrange 2} is a direct consequence of Lagrange multiplier theorem. Next, testing \eqref{lagrange 2} with $\bar{u}_\tau$ and combining with \eqref{pohozaev 2}, \eqref{C.3}, \eqref{C.4}, \eqref{C.19}, \eqref{C.20} and $M(u_\tau)=1$ we obtain
%\begin{align}
%\theta_\tau=-\frac{2(\alpha d-2\alpha-4)}{\alpha d-4}((2\pi)^n \widehat{\mu}^{\infty}_{(2\pi)^{-n}})+o_\tau(1)
%\end{align}
%and the desired claim follows.
This is similarly proved as for Lemma \ref{lem 4.2} and we omit here the repeating arguments.
\end{proof}

%Mimicking the proof of Lemma \ref{lem 3.3}, we have the following result.

\begin{lemma}
	There exists some $v \in \widehat{S}((2\pi)^{-n})$ such that, up to subsequence and $\R^d$-translations, $u_\tau$ converges strongly in $H_{x,y}^1$ to $\rho$ as $\tau\to\infty$, $\widehat{E}^{\infty}(\rho)=\widehat{\mu}^{\infty}_{(2\pi)^{-n}}$ and $\rho$ satisfies
	\begin{align}
		\beta\Delta_x \rho+\hat{\theta}\rho=|\rho|^\alpha \rho.\label{lagrange w}
	\end{align}
\end{lemma}

\begin{proof}
This is similarly proved as for Lemma \ref{lem 3.3} and we omit here the repeating arguments.
\end{proof}

\begin{lemma}\label{lemma no dependence 2}
	{Let $\beta>0$, $\alpha\in(0,\frac{4}{d+n})$. Then} there exists some $\tau_*>0$ such that for all $\tau>\tau_*$ we have $\nabla_y u_\tau=0$.
\end{lemma}

\begin{proof}
%Define $w_\tau:=\sqrt{-\Delta_y}u_\tau$. Then applying $-\Delta_y$ to \eqref{lagrange 2} we obtain
	%\begin{align}
		%\tau\Delta_y^2 w_\tau + \tau\beta\Delta_y w_\tau + \frac{1}{\tau}\Delta^2_{x}w_\tau+\beta\Delta_x w_\tau+\theta_\tau w_\tau=\sqrt{-\Delta_y}(|u_\tau|^\alpha u_\tau).\label{C.26}
	%\end{align}
	%Similar to the proof of Lemma \ref{lemma no dependence}, let $F(z):=|z|^{\alpha}z$, $G_\tau(x,y):=F_z(\rho(x))u_\tau(x,y)+F_{\bar{z}}(\rho(x))\overline{u_\tau(x,y)}$. Then testing \eqref{C.26} with $\bar{w}_\tau$ yields
	%\begin{align}
		%0&=\tau\|\Delta_y w_\tau\|_2^2 + (\tau-1)\beta\|\nabla_y w_\tau\|_2^2 + \frac{1}{\tau}\|\Delta_x w_\tau\|_2^2 -\int_{\R^d\times\T^n}
		%(\sqrt{-\Delta_y}G_\tau)w_\tau\,dxdy\nonumber\\
		%&+\beta\|\nabla_{x,y}w_\tau\|_2^2+\hat{\theta}M(w_\tau)+\int_{\R^d\times\T^n}\sqrt{-\Delta_y}(G_\tau-|u_\tau|^\alpha u_\tau)w_\tau\,dxdy\nonumber\\
		%&+(\theta_\tau-\hat{\theta})M(w_\tau) \nonumber\\
		%&\geq(\tau-1)\beta\|\nabla_y w_\tau\|_2^2 -\int_{\R^d\times\T^n}
		%(\sqrt{-\Delta_y}G_\tau)w_\tau\,dxdy\nonumber\\
		%&+\beta\|\nabla_{x,y}w_\tau\|_2^2+\hat{\theta}M(w_\tau)+\int_{\R^d\times\T^n}\sqrt{-\Delta_y}(G_\tau-|u_\tau|^\alpha u_\tau)w_\tau\,dxdy\nonumber\\
		%&+(\theta_\tau-\hat{\theta})M(w_\tau)=:I+II+III,
	%\end{align}
	%where the terms $I$, $II$ and $III$ are exactly the terms $I_j$, $II_j$ and $III_j$ respectively in the proof of \cite[Lem. 3.6]{TTVproduct2014}, and from therein we conclude that
	%$$ I+II+III\gtrsim \|w_\tau\|_{H_{x,y}^1}^2(1-o_\tau(1)).$$
	%Thus the desired claim follows by taking $\tau$ sufficiently large.
This is similarly proved as for Lemma \ref{lemma no dependence} and we omit here the repeating arguments.
\end{proof}

We are now ready to give the proof of Theorem \ref{thm dependence}, Case (ii).

\begin{proof}[Proof of Theorem \ref{thm dependence}, Case (ii)]
	
	First, Lemma \ref{lemma no dependence 2} yields the existence of $\tau_* \in [0,\infty)$ such that for all $\tau \in (\tau_*,\infty)$ any minimizer $u_\tau$ of $\mu_{1,\tau}$ satisfies $\nabla_y u_\tau = 0$. This implies that for all $\tau \in (\tau_*,\infty)$ we have $\mu_{1,\tau}=(2\pi)^n \widehat{\mu}^\tau_{(2\pi)^{-n}}$. By the continunity of $\mu_{1,\tau}$ and $\widehat{\mu}^\tau_{(2\pi)^{-n}}$ with respect to $\tau$, we know that $\mu_{1,\tau_*}=(2\pi)^n \widehat{\mu}^{\tau_*}_{(2\pi)^{-n}}$.
	
	We now use a scaling argument to finish the proof. For $c>0$ let $\gamma_c:=c^{\frac{1}{d-\frac4\alpha}}$. For $\theta>0$ define the operator $S_\theta u$ by
	$$S_\theta u(x,y):=\theta^{\frac{2}{\alpha}}u(\theta x,y).$$
	Then $u\mapsto S_{\gamma_c} u$ defines a bijection between $S(c)$ and $S(1)$. Direct calculation also shows that $\mH(u)=c^{\frac{d-4/\alpha-2}{d-4/\alpha}}\mH^{\tau}(S_{\gamma_c} u)$ for $\tau = \gamma_c^{2}$, thus $m_{c}=c^{\frac{d-4/\alpha-2}{d-4/\alpha}}\mu_{1,\tau}$ for $\tau = \gamma_c^{2}$. By same rescaling arguments we also infer that {$\widehat{\mu}^1_{(2\pi)^{-n}c}=c^{\frac{d-4/\alpha-2}{d-4/\alpha}}\widehat{\mu}^\tau_{(2\pi)^{-n}}$}
for $\tau = \gamma_c^{2}$. Notice also that the mapping $c\mapsto \gamma_c$ is strictly monotone decreasing on $(0,\infty)$. Thus we know that there exists some $c_{>0}\in(0,\infty)$ (that $c_{>0}<\infty$ follows from the proof of Theorem \ref{thm dependence}, Case (iii) given below, which is also independent of the present proof) such that for any $c\in(0,c_{>0}]$ we have {$m_{c}=(2\pi)^n\widehat{\mu}^1_{(2\pi)^{-n}c}$} and for any $c\in(0,c_{>0})$ it holds $\nabla_y u_c= 0$. This completes the desired proof.
\end{proof}

\subsection{Proof of Theorem \ref{thm dependence}, Case (iii)}
In this subsection we give the proof of the last statement (iii) of Theorem \ref{thm dependence}.

\begin{proof}[Proof of Theorem \ref{thm dependence}, Case (iii)]
We first consider the case $\beta>0$. Recall the quantitiy $\widehat{m}_c^\ld$ defined by \eqref{1.16}. By rescaling we have
$\wm_{(2\pi)^{-n}c}=c^{\frac{d-8/\alpha-4}{d-8/\alpha}}\wm_{(2\pi)^{-n}}^{c^{\frac{4}{ d-8/\alpha}}}$. Using Theorem \ref{1thm nongative} we know that $\widehat{m}_{\|\rho\|^{-2n}_{L_y^2}}^\ld<0$ and $\widehat{m}_{\|\rho\|^{-2n}_{L_y^2}}^\ld$ has an optimizer $Q_\ld$ for all $0<\ld\ll 1$, where $\rho$ is the function constructed in Lemma \ref{lem 3.5}. Since for fixed $c>0$ the mapping $\ld\mapsto \widehat{m}_c^\ld$ is monotone increasing, there exists some $A>0$ and $\ld_0\in(0,\infty)$ such that $\widehat{m}_{\|\rho\|^{-2n}_{L_y^2}}^\ld\leq -A$ for all $\ld\in(0,\ld_0)$. Arguing as in \eqref{2.7} we then infer that there exists some $C_A>0$ such that
\begin{align}\label{3.43}
\inf_{\ld\in(0,\ld_0)}\|Q_\ld\|_{L_x^{\alpha+2}}^{\alpha+2}\geq C_A.
\end{align}
In view of Lemma \ref{lem 3.5}, in order to prove the claim, it suffices to show that $m_{1,\ld}<(2\pi)^n \widehat{m}_{(2\pi)^{-n}}^\ld$ for all $\ld\ll 1$. Again, as in the proof of Lemma \ref{lem 3.5} we define the function $\psi$ by
\begin{align}
\psi(x,y)&:=Q_\ld(x)(\|\rho\|_{L_y^2}/\|\rho_\vare\|_{L_y^2})^{n}\prod_{j=1}^n\rho_\vare(y_j)\nonumber\\
&=:Q_\ld(x)(\|\rho\|_{L_y^2}/\|\rho_\vare\|_{L_y^2})^{n}\Psi_\vare(y).
\end{align}
Before proceeding, we need to adjust the parameters $a,b$ given in the construction of $\rho$. This time we fix
$$b=\bg[2\frac{(\alpha+3)(\pi-a)^{-(\alpha+1)}}{3}\bg]^{\frac{1}{\alpha}}$$ so that $\|\rho\|_{L_y^2}^2<\|\rho\|_{L^{\alpha+2}_y}^{\alpha+2}$. Moreover, we have
$$\|\rho\|_{L_y^2}^2=\frac{2^{1+\frac{2}{\alpha}}}{3}\bg(\frac{\alpha+3}{3}\bg)^{\frac{2}{\alpha}}(\pi-a)^{1+\frac{2}{\alpha}}. $$
Consequently,
\begin{align}
m_{1,\ld}\leq E_\ld(\psi)&=\|\rho\|_{L_y^2}^{2n}\bg(\frac{1}{2}\|\Delta_x Q_\ld\|_{L_x^2}^2+\frac{\beta\sqrt{\ld}}{2}\|\nabla_x Q_\ld\|_{L_x^2}^2
-\frac{1}{\alpha+2}\|Q_\ld\|_{L_x^{\alpha+2}}^{\alpha+2}\bg)\nonumber\\
&-\frac{\|\rho\|_{L_y^2}^{2n}}{\alpha+2}
\bg(\|\rho\|_{L_y^2}^{-2n}\|\rho\|_{L_y^2}^{(\alpha+2)n}\|\rho_\vare\|_{L_y^{2}}^{-(\alpha+2)n}\|\rho_\vare\|_{L_y^{\alpha+2}}^{(\alpha+2)n}-1\bg)
\|Q_\ld\|_{L_x^{\alpha+2}}^{\alpha+2}\nonumber\\
&+\frac{\ld}{2}\|\rho\|_{L_y^2}^{-2n}\bg(\frac{\|\rho\|_{L_y^2}^2}{\|\rho_\vare\|_{L_y^2}^2}\bg)^n(2\pi)^n(\|\Delta_y\Psi_\vare\|_{L_y^2}^2
+\beta\|\nabla_y \Psi_\vare\|_{L_y^2}^2)\nonumber\\
&=:\|\rho\|_{L_y^2}^{2n}\widehat{m}^\ld_{\|\rho\|_{L_y^2}^{-2n}}+I_1+I_2.
\end{align}
We now proceed as follows: first, as in Lemma \ref{lem 3.5}, the function $c\mapsto c^{-1}m_c^\ld$ is strictly decreasing on $(0,\infty)$, thus we choose some $a$ sufficiently close to $\pi$ so that $\|\rho\|_{L_y^2}^2<2\pi$ and thus also $\|\rho\|_{L_y^2}^{2n}\widehat{m}^\ld_{\|\rho\|_{L_y^2}^{-2n}}
<(2\pi)^n m_{(2\pi)^{-n}}^\ld$. We then fix this choice of $a$. Next, since $\|\rho\|_{L_y^2}^2<\|\rho\|_{L^{\alpha+2}_y}^{\alpha+2}$, we have
\begin{align}
\|\rho\|_{L_y^2}^{-2n}\|\rho\|_{L_y^2}^{(\alpha+2)n}\|\rho_\vare\|_{L_y^{2}}^{-(\alpha+2)n}\|\rho_\vare\|_{L_y^{\alpha+2}}^{(\alpha+2)n}-1
=\frac{\|\rho\|_{L_y^2}^{(\alpha+2)n}(\|\rho\|^{(\alpha+2)n}_{L_y^{\alpha+2}}+o_\vare(1))}
{\|\rho\|_{L_y^2}^{2n}(\|\rho\|^{(\alpha+2)n}_{L_y^2}+o_\vare(1))}-1 \geq \zeta>0\label{3.46}
\end{align}
for some $\zeta>0$ and all $\vare\ll 1$. We then fix some $\vare$ so that \eqref{3.46} holds and combining with \eqref{3.43}
\begin{align}
I_1\leq -\frac{\|\rho\|_{L_y^2}^{2n}C_A\zeta}{\alpha+2}.
\end{align}
Finally, for fixed $a$ and $\vare$ we see that $\lim_{\ld\to 0}I_2=0 $. The desired claim follows by taking $\ld$ small.

We finally consider the case $\beta<0$. In view of the proof in the case $\beta>0$, it suffices to show that \eqref{3.43} also holds in the case $\beta<0$. From Lemma \ref{lem 2.5} we also know that
\begin{align}
\inf_{u\in \widehat{S}(\|\rho\|^{-2n}_{L_y^2})}(\|\Delta_{x}u\|_{2}^2+\beta \sqrt{\ld}\|\nabla_{x} u\|_2^2)= -\frac{\beta^2\ld \|\rho\|^{-2n}_{L_y^2}}{4}.\label{3.46a}
\end{align}
Since $\beta<0$, we also have $\widehat{m}_{\|\rho\|^{-2n}_{L_y^2}}^{\ld}\leq \widehat{m}^0_{\|\rho\|^{-2n}_{L_y^2}}<0$, where  $\widehat{m}^0_{\|\rho\|^{-2n}_{L_y^2}}<0$ follows from Theorem \ref{1thm nongative} (ia). Thus estimating as in \eqref{2.25} we conclude that
\begin{align}
\liminf_{\ld\to 0}\|Q_\ld\|_{\alpha+2}^{\alpha+2}\gtrsim
\liminf_{\ld\to 0}(-\frac{1}{4}\beta^2\ld \|\rho\|^{-2n}_{L_y^2}-\|\rho\|^{2n}_{L_y^2}\widehat{m}^0_{\|\rho\|^{-2n}_{L_y^2}})
=-\|\rho\|^{2n}_{L_y^2}\widehat{m}^0_{\|\rho\|^{-2n}_{L_y^2}}>0.
\end{align}
This completes the proof.
\end{proof}

\section{Proof of Theorem \ref{thm stability}}
\begin{proof}[Proof of Theorem \ref{thm stability}]
Assume on the contrary that Theorem \ref{thm stability} does not hold. Then there exist $(\phi_0^j)_j\subset H^2_{x,y}$ and $(t_j)_j\subset \R$ such that
\begin{gather}
\lim_{j\to\infty}\mathrm{dist}_{H_{x,y}^2}(\phi_0^j,\Gamma_c)=0,\label{4.1}\\
\liminf_{j\to\infty}\mathrm{dist}_{H_{x,y}^2}(\phi^j(t_j),\Gamma_c)>0,\label{4.2}
\end{gather}
where $\phi_j$ is the global solution of \eqref{nls} with $\phi_j(0)=\phi_0^j$. By \eqref{4.1} we have
\begin{align}
M(\phi_0^j)=c+o_j(1),\quad E(\phi_0^j)=m_c+o_j(1).
\end{align}
By conservation of mass and energy we infer that
\begin{align}
M(\phi_j)=c+o_j(1),\quad E(\phi_j)=m_c+o_j(1).
\end{align}
By fundamental perturbation arguments, for $\tilde{\phi}_j:=\sqrt{c}M(\phi_j)^{-\frac{1}{2}}\phi_j\in S(c)$ we have
\begin{align}
M(\tilde{\phi}_j)=c+o_j(1),\quad E(\tilde{\phi}_j)=m_c+o_j(1).
\end{align}
This implies that the sequence $(\tilde{\phi}_j)_j$ is a minimizing sequence of $m_c$. From the proof of Theorem \ref{thm existence} we know that up to a subsequence, $\tilde{\phi}_j$ converges strongly to a minimizer $u_c$ of $m_c$ in $H_{x,y}^2$, which contradicts \eqref{4.2}.
\end{proof}
\subsubsection*{Acknowledgements}
Y.L. is funded by Deutsche Forschungsgemeinschaft (DFG) through the Priority Programme SPP-1886 (No. NE 21382-1).

%\addcontentsline{toc}{section}{References}
%\bibliography{biharmonic}
%\bibliographystyle{acm}

\end{document}